\documentclass[12pt, stslayout, noinfoline, reqno, preprint]{imsart}

\usepackage{graphicx}
\usepackage{amsthm, amssymb, amsmath}
\usepackage{epstopdf}
\usepackage{natbib}
\usepackage{subfigure}
\usepackage{algcompatible}
\usepackage[floatrow]{trivfloat}
\usepackage{comment}
\usepackage{tikz}
\usepackage{graphicx}
\usepackage[hmarginratio=1:1,top=32mm,columnsep=20pt]{geometry} 
\usepackage[backref=page, colorlinks=true, urlcolor=blue, pdfborder={0 0 0}]{hyperref}
\usepackage{cleveref}

\startlocaldefs
\numberwithin{equation}{section}
\theoremstyle{plain}

\newtheorem{theorem}{Theorem}
\newtheorem{lemma}[theorem]{Lemma}
\newtheorem{proposition}[theorem]{Proposition}
\newtheorem{corollary}[theorem]{Corollary}
\newtheorem{definition}[theorem]{Definition}
\newtheorem{remark}[theorem]{Remark}

\newcommand{\N}{\mathsf{N}}
\newcommand{\E}{\mathbb{E}}
\newcommand{\G}{\mathcal{G}}
\newcommand{\R}{\mathbb{R}}

\newcommand{\CC}{\mathcal{C}}

\newcommand{\rd}{\mathrm{d}}
\DeclareMathOperator{\Tr}{Tr}
\endlocaldefs
\newcommand{\bC}{\mathsf{C}}

\newcommand{\Z}{\mathbb{Z}}

\newcommand{\supp}{{\rm supp\ }}
\newcommand{\KL}{{\rm KL}}
\newcommand{\mm}{\mathfrak{m}}
\newcommand{\mC}{\mathfrak{C}}

\newcommand{\mP}{{\mathcal P}}

\newcommand{\mF}{{\mathcal F}}

\newcommand{\mI}{{\mathcal I}}

\newcommand{\Tp}{T_\rho\mP_+}
\newcommand{\cTp}{T_\rho^*\mP_+}
\newcommand{\Ons}{V_{\rho,\CC}}
\newcommand{\Onsop}{\Delta_{\rho,\CC}}

\newcommand{\nc}{\normalcolor}

\graphicspath{{./}{./img/}}

\begin{document}

\begin{frontmatter}
    \title{Interacting Langevin Diffusions: Gradient Structure And Ensemble Kalman Sampler}
    \runtitle{Ensemble Kalman Sampler}
    \begin{aug}
        \author{\fnms{Alfredo} \snm{Garbuno-Inigo}\ead[label=e1]{agarbuno@caltech.edu}},
        \author{\fnms{Franca} \snm{Hoffmann}\ead[label=e2]{fkoh@caltech.edu}},
        \author{\fnms{Wuchen} \snm{Li}\ead[label=e3]{wcli@math.ucla.edu}}
        \and
        \author{\fnms{Andrew M.} \snm{Stuart}\ead[label=e4]{astuart@caltech.edu}}

        \runauthor{Garbuno-Inigo, Hoffmann, Li \& Stuart}

        \address[cal]{Department of Computing and Mathematical Sciences, Caltech, Pasadena, CA.
        \printead*{e1,e2,e4}
        }

        \address[ucla]{Department of Mathematics, UCLA, Los Angeles, CA.
        \printead*{e3}
        }
    \end{aug}

    \maketitle

    \begin{abstract}
      Solving inverse problems without the use of derivatives or adjoints of the
      forward model is highly desirable in many applications arising in science and
      engineering. In this paper we propose a new version of such a methodology, a
      framework for its analysis, and numerical evidence of the practicality of the
      method proposed. Our starting point is an ensemble of over-damped Langevin
      diffusions which interact through a single preconditioner computed as the
      empirical ensemble covariance. We demonstrate that the nonlinear Fokker-Planck
      equation arising from the mean-field limit of the associated stochastic
      differential equation (SDE) has a novel gradient flow structure, built on the
      Wasserstein metric and the covariance matrix of the noisy flow. Using this
      structure, we investigate large time properties of the Fokker-Planck equation,
      showing that its invariant measure coincides with that of a single Langevin
      diffusion, and demonstrating exponential convergence to the invariant measure in
      a number of settings. We introduce a new noisy variant on ensemble Kalman
      inversion (EKI) algorithms found from the original SDE by replacing exact
      gradients with ensemble differences; this defines the ensemble Kalman sampler
      (EKS). Numerical results are presented which demonstrate its efficacy as a
      derivative-free approximate sampler for the Bayesian posterior arising from
      inverse problems.
    \end{abstract}
\end{frontmatter}
\section{Problem Setting}

\subsection{Background}

Consider the inverse problem of finding $u \in \R^d$ from $y \in \R^K$ where
\begin{align}
\label{eq:IP}
    y = \G(u) + \eta,
\end{align}
$\G:\R^d\to\R^K$ is a known non-linear forward operator
and $\eta$ is the unknown observational
noise. Although $\eta$ itself is unknown, we assume that
it is drawn from a known
probability distribution; to be concrete we assume that this
distribution is a centered Gaussian:
$\eta \sim \N(0,\Gamma)$ for a known covariance  matrix $\Gamma\in\R^{K\times
K}$. In summary, the objective of the
inverse problem is to find information about
the truth $u^\dagger$ underlying the data $y$; the forward map $\G$, the covariance
$\Gamma$ and the data $y$ are all viewed as given.

A key role in any optimization scheme to solve \eqref{eq:IP} is played by
$\ell(y,\G(u))$ for some loss function $\ell: \R^K \times \R^K \mapsto \R.$
For additive Gaussian noise the natural loss function
is \footnote{For any positive-definite symmetric matrix $A$ we define $\langle
a,a' \rangle_{A}=\langle a,A^{-1}a' \rangle= \langle
A^{-\frac12}a,A^{-\frac12}a' \rangle$ and $\|a\|_{A}=\|A^{-\frac12}a\|.$}
$$\ell(y,y')=\frac12\|y-y'\|_{\Gamma}^2,$$ leading to the nonlinear
least squares functional
\begin{equation}\label{eq:phi}
    \Phi(u)=\frac{1}{2}\|y-\G(u)\|^2_{\Gamma}.
\end{equation}

In the Bayesian approach to inversion \citep{kaipio2006statistical} we place a prior distribution on the unknown $u$, with Lebesgue density $\pi_0(u)$, then the posterior
density on $u|y$, denoted $\pi(u)$, is given by
\begin{equation}
\label{eq:post}
 \pi(u) \propto \exp\bigl(-\Phi(u)\bigr)\pi_0(u).
\end{equation}
In this paper we will concentrate on the case where the prior is a centred Gaussian $\N(0,\Gamma_0)$, assuming throughout that $\Gamma_0$ is strictly positive-definite and hence invertible. If we define
\begin{equation}
\label{eq:R}
R(u)=\frac12 \|u\|_{\Gamma_0}^2
\end{equation}
and
\begin{equation}
    \label{eq:phir}
\Phi_R(u)=\Phi(u)+R(u),
\end{equation}
then
\begin{equation}
\label{eq:post2}
 \pi(u) \propto \exp\bigl(-\Phi_R(u)\bigr).
\end{equation}
Note that the regularization $R(\cdot)$
is of Tikhonov-Phillips form \citep{engl1996regularization}.

Our focus throughout is on using interacting particle systems to approximate
Langvein-type stochastic dynamical systems to sample from \eqref{eq:post2}.
Ensemble Kalman inversion (EKI), and variants of it, will be central in our
approach because these methods play an important role in large-scale scientific
and engineering applications in which it is undesirable, or impossible, to
compute derivatives and adjoints defined by the forward map $\G$. Our goal is to
introduce a noisy version of EKI which may be used to generate approximate
samples from \eqref{eq:post2} based only on evaluations of $\G(u)$, to exemplify
its potential use and to provide a framework for its analysis. We refer to the
new methodology as \emph{ ensemble Kalman sampling}  (EKS).

\subsection{Literature Review}

The overdamped Langevin equation provides the simplest example of a reversible
diffusion process with the property that it is invariant with respect to
\eqref{eq:post2} \citep{pavliotis2014stochastic}. It provides a conceptual
starting point for a range of algorithms designed to draw approximate samples
from the density \eqref{eq:post2}. This idea may be generalized to
non-reversible diffusions such as those with state-dependent noise
\citep{duncan2016variance}, those which are higher order in time
\citep{ottobre2011asymptotic} and combinations of the two
\citep{girolami2011riemann}. In the case of higher order dynamics the desired
target measure is found by marginalization. There are also a range of methods,
often going under the collective names Nos\'e-Hoover-Poincar\'e, which identify
the target measure as the marginal of an invariant measure induced by (ideally)
chaotic and mixing deterministic dynamics \citep{leimkuhler2016molecular} or a
mixture between chaotic and stochastic dynamics \citep{leimkuhler2009gentle}.
Furthermore, the Langevin equation may be shown to govern the behaviour of a
wide range of Monte Carlo Markov Chain (MCMC) methods; this work was initiated
in the seminal paper \citep{roberts1997} and has given rise to many related works
\citep{roberts1998optimal,roberts2001optimal,bedard2007weak,bedard2008optimal,bedard2008optimalb,mattingly2012diffusion,pillai2014noisy,ottobre2011asymptotic};
for a recent overview see \citep{yang2019optimal}.

In this paper we will introduce an interacting particle system generalization of
the overdamped Langevin equation, and use ideas from ensemble Kalman methodology
to generate approximate solutions of the resulting stochastic flow, and hence
approximate samples of \eqref{eq:post2}, without computing dervatives of the log
likelihood.  The ensemble Kalman filter was originally introduced as a method
for state estimation, and later extended as the EKI to the solution of general
inverse problems and parameter estimation problems. For a historical development
of the subject, the reader may consult the books
\citep{evensen2009data,oliver2008inverse,majda2012filtering,law2015data,reich2015probabilistic}
and the recent review \citep{carrassi2018data}. The Kalman filter itself was
derived for linear Gaussian state estimation problems \citep{kalman1960new,KB}.
In the linear setting, ensemble Kalman based methods may be viewed as Monte
Carlo approximations of the Kalman filter; in the nonlinear case ensemble Kalman
based methods do not converge to the filtering or posterior distribution in the
large particle limit \citep{ernst2015analysis}. Related interacting particle
based methodologies of current interest include Stein variational gradient
descent \citep{lu2018scaling, Liu2016,detommaso2018stein} and the Fokker-Planck
particle dynamics of Reich \citep{reich2018data,pathiraja2019discrete}, both of
which map an arbitrary initial measure into the desired posterior measure over
an infinite time horizon $s \in [0,\infty)$. A related approach is to introduce
an artificial time $s \in [0,1]$ and a homotopy between the prior at time $s=0$
and the posterior measure at time $s=1$ and write an evolution equation for the
measures \citep{daum2011particle,reich2011dynamical,el2012bayesian,LMMR}; this
evolution equation can be approximated by particle methods. There are also other
approaches in which optimal transport is used to evolve a sequence of particles
through a transportation map \citep{reich2013,Marzouk2016} to solve probabilistic
state estimation or inversion problems as well as interacting particle systems
designed to reproduce the solution of the filtering problem
\citep{crisan2010approximate,yang2013feedback}. The paper \citep{del2018stability}
studies ensemble Kalman filters from the perspective of the mean-field process,
and propagation of chaos. Also of interest are the consensus-based optimization
techniques given a rigorous setting in \citep{carrillo2018analytical}.

The idea of using interacting particle systems derived from coupled
Langevin-type equations is introduced within the context of MCMC methods in
\citep{leimkuhler2018ensemble}; these methods require computation of derivatives
of the log likelihood. In work \citep{Duncan}, concurrent with this paper, the
interacting Langevin diffusions \eqref{eq:eki-basic-iter},\eqref{eq:ec} below
are studied, the goal being to demonstrate that the pre-conditioning removes
slow relaxation rates when they are present in the standard Langevin equation
\eqref{eq:ble}; such a result is proven in the case where the potential $\Phi_R$
is quadratic and the posterior measure of interest is Gaussian. A key concept
underlying both \citep{leimkuhler2018ensemble} and \citep{Duncan} is the idea of
finding algorithms which converge to equilibrium at rates independent of the
conditioning of the Hessian of the log posterior, an idea introduced in the
affine invariant samplers of \cite{goodman2010ensemble}.  \nc

Continuous-time limits of ensemble Kalman filters for state estimation were
first introduced and studied systematically in the papers
\citep{bergemann2012ensemble,reich2011dynamical,bergemann2010localization,bergemann2010mollified};
the papers \citep{bergemann2010localization,bergemann2010mollified} studied the
``analysis'' step of filtering (using Bayes theorem to incorporate data) through
introduction of an artificial continuous time; the papers
\citep{bergemann2012ensemble,reich2011dynamical} developed a seamless framework
that integrated the true time for state evolution and the artificial time for
incorporation of data into one. The resulting methodology has been studied in a
number of subsequent papers, see
\citep{del2017stability,del2018stability,de2018long,taghvaei2018kalman} and the
references therein. A slightly different seamless continuous time formulation
was introduced, and analyzed, a few years later in
\citep{law2015data,kelly2014well}. Continuous time limits of ensemble methods for
solving inverse problems were introduced and analyzed in the paper
\citep{schillings2017analysis}; in fact the work in the papers
\citep{bergemann2010localization,bergemann2010mollified} can be re-interpreted in
the context of ensemble methods for inversion and also results in similar, but
slightly different continuous time limits. The idea of iterating ensemble
methods to solve inverse problems originated in the papers
\citep{chen2012ensemble,emerick2013investigation}, which were focussed on
applications in oil-reservoir applications; the paper
\citep{iglesias2013ensemble} describes, and demonstrated the promise of, the
methods introduced in those papers for quite general inverse problems. The
specific continuous time version of the methodology, which we refer to as EKI in
this paper, was identified in \citep{schillings2017analysis}.

There has been significant activity devoted to the gradient flow structure
associated with the Kalman filter itself.  A well-known result is that for a
constant state process, Kalman filtering is the gradient flow with respect to
the Fisher-Rao metric \citep{LMMR, HG,Ollivier2017_online}. It is worth noting
that the Fisher-Rao metric connects to the covariance matrix, see details in
\citep{IG2}. On the other hand, optimal transport \citep{vil2008} demonstrates the
importance of the $L^2$-Wasserstein metric in probability density space. The
space of densities equipped with this metric introduces an infinite-dimensional
Riemannian manifold, called the density manifold \citep{Lafferty,otto2001,LiG}.
Solutions to the Fokker-Planck equation are gradient flows of the relative
entropy in the density manifold \citep{otto2001,JKO}. Designing time-stepping
methods which preserve gradient structure is also of current interest: see
\citep{pathiraja2019discrete} and, within the context of Wasserstein gradient
flows, \citep{LiM,tong2018wasserstein,LiM2}. The subject of discrete gradients
for time-integration of gradient and Hamiltonian systems is developed in
\citep{humphries1994runge,gonzalez1996time,mclachlan1999geometric,hairer2013energy}.
Furthermore, the papers \citep{schillings2017analysis,schillings2018convergence}
study continuous time limits of EKI algorithms and, in the case of linear
inverse problems, exhibit a gradient flow structure for the standard least
squares loss function, preconditioned by the empirical covariance of the
particles; a related structure was highlighted in
\citep{bergemann2010localization}. The paper \citep{herty2018kinetic}, which has
inspired aspects of our work, builds on the paper \citep{schillings2017analysis}
to study the same problem in the mean-field limit; their mean-field perspective
brings considerable insight which we build upon in this paper. Recent work
\citep{ding2019mean} has studied the approach to the mean-field limit for linear
inverse problems, together with making connection to the appropriate nonlinear
Fokker-Planck equation whose solution characterizes the distribution in the
mean-field limit.

In this paper, we study a new noisy version of EKI,  the ensemble Kalman sampler
(EKS), and related mean-field limits, the aim being the construction of methods
which lead to approximate posterior samples, without the use of adjoints, and
overcoming the issue that the standard noisy EKI does not reproduce the
posterior distribution, as highlighted in \citep{ernst2015analysis}. We emphasize
that the practical derivative-free algorithm that we propose rests on a
particle-based approximation of a specific preconditioned gradient flow, as
described in section 4.3 of the paper \citep{KovachkiStuart2018_ensemble}; we add
a judiciously chosen noise to this setting and it is this additional noise which
enables approximate posterior sampling. Related approximations are also studied
in the paper \citep{pathiraja2019discrete} in which the effect of both
time-discretization and particle approximation are discussed when applied to
various deterministic interacting particle systems with gradient structure. In
order to frame the analysis of our methods, we introduce a new metric, named the
Kalman-Wasserstein metric, defined through both the covariance matrix of the
mean field limit and the Wasserstein metric. The work builds on the novel
perspectives introduced in \citep{herty2018kinetic} and leads to new algorithms
that will be useful within large-scale parameter learning and uncertainty
quantification studies, such as those proposed in \citep{schneider2017earth}.

\subsection{Our Contribution}
The contributions in this paper are:
\begin{itemize}
    \item  We introduce a new noisy perturbation of the continuous time
    ensemble Kalman inversion (EKI) algorithm, leading to an interacting
particle system in  stochastic
differential equation (SDE) form, the ensemble Kalman sampler (EKS).

\item  We also introduce a related SDE, in which ensemble differences are approximated by gradients; this approximation is exact for linear
inverse problems. We study the mean-field limit of this related SDE, and
    exhibit a novel Kalman--Wasserstein gradient flow structure in the associated nonlinear Fokker-Planck equation.

    \item Using this Kalman--Wasserstein structure we characterize the steady states of the
    nonlinear Fokker-Planck equation, and show that one of them is the posterior
    density \eqref{eq:post2}.

    \item By explicitly solving the nonlinear Fokker-Planck equation in the case of linear $\G$,
    we demonstrate that the posterior density is a global attractor for all initial densities of finite energy which
    are not a Dirac measure.

    \item We provide numerical examples which demonstrate that
the EKS algorithm gives good approximate samples from the posterior
distribution for both a simple low dimensional test problem, and for a PDE
inverse problem arising in Darcy flow.

\end{itemize}

In Section \ref{sec:DSS} we introduce the various stochastic dynamical systems which form the basis for the
proposed methodology and analysis: Subsection \ref{ssec:Lan} describes an interacting particle system
variant on Langevin dynamics; Subsection \ref{ssec:EKI}  recaps the EKI
methodology, and describes the SDE arising in the case when the data
is perturbed with noise; and Subsection \ref{ssec:EKS} introduces the new noisy EKS algorithm, which arises from perturbing the particles with noise, rather than perturbing the
data. In Section \ref{sec:TP} we discuss the theoretical properties
underpinning the proposed new methodology and in Section
\ref{sec:NE} we describe numerical results which demonstrate the
value of the proposed new methodology. We conclude in Section \ref{sec:C}.

\section{Dynamical Systems Setting}
\label{sec:DSS}

This section is devoted to the various noisy dynamical systems that underpin the
paper: in the three constituent subsections we introduce an interacting particle
version of Langevin dynamics, the EKI algorithm and the new EKS algorithm. In so
doing, we introduce a sequence of continuous time problems that are designed to
either maximise the posterior distribution $\pi(u)$ (EKI), or generate
approximate samples from the posterior distribution $\pi(u)$ (noisy EKI and the
EKS). We then make a linear approximation within part of the EKS and take the
mean-field limit leading to a novel nonlinear Fokker-Planck equation studied in
the next section.

\subsection{Variants On Langevin Dynamics}
\label{ssec:Lan}

The overdampled Langevin equation has the form
\begin{align}
\label{eq:ble}
\dot{u} = - \nabla \Phi_R(u) + \sqrt{2} \,  \dot{\textbf{W}}\,;
\end{align}
where $\textbf{W}$ denotes a standard Brownian motion in $\mathbb{R}^d.$\footnote{In this SDE, and all that follow,
the rigorous interpretation is through the It\^o integral formulation of the problem.}
References to the relevant literature may be found in the introduction.
A common approach to speed up convergence is to introduce a symmetric matrix $\bC$ in the corresponding gradient descent scheme,
\begin{align} \label{eq:eki-basic}
\dot{u} = - \bC\nabla \Phi_R(u) + \sqrt{2 \, \bC} \,  \dot{\textbf{W}}\,.
\end{align}
The key concept behind this stochastic dynamical system is that, under
conditions on $\Phi_R$ which ensure ergodicity, \emph{an arbitrary initial
distribution is transformed into the desired posterior distribtion
over an infinite time horizon.}

To find a suitable matrix $\bC\in\mathbb{R}^{d\times d}$ is of general interest. We propose to evolve an interacting set of particles $U=\{u^{(j)}\}_{j = 1}^J$ according to the following system of SDEs:
\begin{align} \label{eq:eki-basic-iter}
\dot{u}^{(j)} = - \bC(U)\nabla \Phi_R(u^{(j)}) + \sqrt{2 \, \bC(U)} \,  \dot{\textbf{W}}^{(j)}\,,
\end{align}
Here, the  $\{\textbf{W}^{(j)}\}$ are a collection of i.i.d. standard Brownian motions in the
space $\R^d$. The matrix $\bC(U)$ depends non-linearly on all ensemble members, and is chosen to be the empirical covariance between particles,
\begin{align}
\label{eq:ec}
\bC(U) &= \frac{1}{J} \sum_{k = 1}^J (u^{(k)} - \bar{u}) \otimes
(u^{(k)} - \bar{u}) \in \R^{d\times d}\,.
\end{align}
where $\bar{u}$ denotes the sample mean
\begin{align*}
 \bar{u} &= \frac{1}{J} \sum_{j = 1}^J u^{(j)}\,.
\end{align*}
This choice of preconditioning is motivated by an underlying gradient flow structure which we exhibit in Section~\ref{ssec:KWGF}. System \eqref{eq:eki-basic-iter} can be re-written as
\begin{align} \label{eq:implement-add}
\dot{u}^{(j)} = - \frac{1}{J}\sum_{k = 1}^J \, \langle D\G(u^{(j)})\bigl(u^{(k)} - \bar{u}\bigr), \G(u^{(j)}) - y \rangle_\Gamma \, u^{(k)} \, - \,
\bC(U) \Gamma_0^{-1}u^{(j)} +\,\sqrt{2\bC(U)} \, \dot{\textbf{W}}^{(j)}.
\end{align}
(We used the fact that it is possible to replace $u^{(k)}$ by $u^{(k)}-\bar{u}$ after the $\Gamma-$weighted inner-product in
\eqref{eq:implement-add} without changing the equation.) We will introduce an ensemble Kalman based methodology to approximate this
interacting particle system, the EKS.

\subsection{Ensemble Kalman Inversion}
\label{ssec:EKI}
To  understand the EKS we first recall the ensemble Kalman inversion (EKI) methodology which can be interpreted as a derivative-free
optimization algorithm to invert $\G$ \citep{iglesias2013ensemble,iglesias2016regularizing}.
The continuous time version of the algorithm is given
by \citep{schillings2017analysis}:
\begin{align} \label{eq:eki_no}
\dot{u}^{(j)} = - \frac{1}{J}\sum_{k = 1}^J \, \langle \G(u^{(k)}) - \bar{\G}, \G(u^{(j)}) - y \rangle_\Gamma \, u^{(k)}\,.
\end{align}
This interacting particle dynamic acts to both drive particles towards consensus and to fit the data.
In \citep{chen2012ensemble,emerick2013investigation} the idea of using ensemble Kalman methods to map prior samples
into posterior samples was introduced (see the introduction for a literature review). Interpreted in our continuous time-setting,
the methodology operates by evolving a noisy set of interacting
particles given by
\begin{align} \label{eq:eki_noise}
\dot{u}^{(j)} = - \frac{1}{J}\sum_{k = 1}^J \, \langle \G(u^{(k)}) - \bar{\G}, \G(u^{(j)}) - y \rangle_\Gamma \, u^{(k)} \,  + \, \bC^{up}(U) \, \Gamma^{-1} \, \sqrt{\Sigma} \, \dot{\textbf{W}}^{(j)},
\end{align}
where the  $\{\textbf{W}^{(j)}\}$ are a collection of i.i.d. standard Brownian motions in the data space $\R^K;$
different choices of $\Sigma$ allow to remove noise and obtain
an optimization algorithm ($\Sigma=0$) or to add noise in a manner
which, for linear problems, creates a dynamic transporting
the prior into the posterior in one time unit ($\Sigma=\Gamma$, see
discussion below).

Here, the operator $\bC^{up}$ denotes the empirical cross
covariance matrix of the ensemble members,
\begin{align}
\bC^{up}(U) &:= \frac{1}{J} \sum_{k = 1}^J (u^{(k)} - \bar{u} ) \otimes \left( \G(u^{(k)}) - \bar{\G} \right) \in \R^{d\times K},\qquad
\bar{\G} := \frac{1}{J} \sum_{k = 1}^J  \G(u^{(k)}).
\end{align}
The approach is designed in the linear case
\emph{to transform prior samples into
posterior samples in one time unit} \citep{chen2012ensemble}. In contrast to Langevin dynamics
this has the desirable property that it works over a single time unit,
rather than over an infinite time horizon. But it is considerably
more rigid as it requires initialization at the prior.
Furthermore, the long time dynamics do not have the desired sampling property,
but rather collapse to a single point, solving the optimization problem
of minimizing $\Phi(u).$ We now demonstrate these points
by considering the linear problem.

To be explicit we consider the case where
\begin{align}
\label{eq:linear}
\G(u) = Au.
\end{align}
In this case, the regularized misfit equals
\begin{align}\label{eq:phi-reg}
    \Phi_R(u) = \frac{1}{2}\|Au - y\|^2_\Gamma + \frac12 \|u\|^2_{\Gamma_0}.
\end{align}
The corresponding gradient can be written as
\begin{gather}
\label{eq:gather}
    \nabla \Phi_R(u)=B^{-1}u - r\,,\\
    r := A^\top\Gamma^{-1} y \in \R^d\,,\quad
    B:= \Bigl(A^\top\Gamma^{-1} A + \Gamma_0^{-1}\Bigr)^{-1} \in\R^{d\times d}.\notag
\end{gather}
The posterior mean is thus $Br$ and the posterior covariance is $B.$

In the linear setting \eqref{eq:linear} and with the
choice $\Sigma=\Gamma$, the EKI algorithm defined in \eqref{eq:eki_noise}
has mean $\mm$ and covariance $\mC$ which satisfy the closed equations
\begin{subequations}
\begin{align}
    \frac{\rd}{\rd t} \mm(t) &= - \mC(t)\bigl(A^\top\Gamma^{-1} A\mm(t) - r\big) \label{eq:mudot0}\\
    \frac{\rd}{\rd t} \mC(t)&=-\mC(t)A^\top\Gamma^{-1} A\mC(t) \label{eq:Cdot0}.
\end{align}
\end{subequations}
These results may be established by similar techniques to those used below in Subsection \ref{ssec:LP}. (A more general analysis of the SDE \eqref{eq:eki_noise},
and its related nonlinear Fokker-Planck equation, is undertaken in
\citep{ding2019mean}.) It follows that
\[ \frac{\rd }{\rd t} \mC(t)^{-1}=-\mC(t)^{-1}\left( \frac{\rd }{\rd t} \mC(t)\right)\mC(t)^{-1}= A^\top \Gamma^{-1}A
\]
and therefore $\mC(t)^{-1}$ grows linearly in time.
If the initial covariance is given by the prior $\Gamma_0$ then
$$\mC(t)^{-1}=\Gamma_0^{-1}+A^\top \Gamma^{-1}A t$$
demonstrating that $\mC(1)$ delivers the posterior
covariance; furthermore it then follows that
$$\frac{d}{dt}\large\{\mC(t)^{-1}\mm(t)\large\}=r$$
so that, initializing with prior mean $\mm(0)=0$ we obtain
$$\mm(t)=\Bigl(\Gamma_0^{-1}+A^\top \Gamma^{-1}A t\Bigr)^{-1}rt$$
and $\mm(1)$ delivers the posterior mean.

The resulting equations  for the mean and covariance are simply those which arise from applying the Kalman-Bucy filter \citep{KB} to  the model
\begin{align*}
    \frac{\rd}{\rd t} u &=0\\
    \frac{\rd}{\rd t} z &:=y= Au+\sqrt{\Gamma} \dot{\textbf{W}},
\end{align*}
where $\textbf{W}$ denotes a standard unit Brownian motion in the data space $\R^K.$
The exact closed form of equations for the first two moments, in the setting of the Kalman-Bucy filter, was established in Section 4 of the paper \citep{reich2011dynamical}
for finite particle approximations, and transfers verbatim to this mean-field setting.

The analysis reveals interesting behaviour in the large
time limit: the covariance shrinks to zero and the mean converges
to the solution of the unregularized least squares problem;
we thus have ensemble collapse and solution of an optimization problem,
rather than a sampling problem. This highights an interesting perspective
on the EKI, namely as an optimization method rather than a sampling
method. A key point to appreciate is that the noise introduced
in \eqref{eq:eki_noise} arises from
the observation $y$ being perturbed with additional noise. In what follows we instead directly
perturb the particles themselves. The benefits of introducing noise on the particles, rather than the data,
was demonstrated in \citep{KovachkiStuart2018_ensemble}, although in that setting only optimization, and not Bayesian
inversion, is considered.

\subsection{The Ensemble Kalman Sampler}
\label{ssec:EKS}

We now demonstrate how to introduce noise on the particles within the ensemble Kalman methodology, with our starting point being
\eqref{eq:implement-add}. This gives the EKS. In contrast to the standard
noisy EKI \eqref{eq:eki_noise}, the EKS is based on a dynamic which
\emph{transforms an arbitrary initial
distribution into the desired posterior distribution, over an infinite
time horizon.}
In many applications, derivatives of the forward map $\G$ are either not available, or extremely costly to obtain. A common technique used in ensemble
Kalman methods is to approximate the gradient $\nabla\Phi_R$ by differences in order to obtain a derivative-free algorithm for inverting $\G$.
To this end, consider the dynamical system \eqref{eq:implement-add} and invoke the approximation
$$D\G(u^{(j)})\bigl(u^{(k)} - \bar{u}\bigr) \approx \bigl(\G(u^{(k)}) - \bar{\G}\bigr).$$
This leads to the following derivative-free algorithm to generate approximate samples from the posterior distribution,
\begin{align} \label{eq:implement}
\dot{u}^{(j)} = - \frac{1}{J}\sum_{k = 1}^J \, \langle \G(u^{(k)}) - \bar{\G}, \G(u^{(j)}) - y \rangle_\Gamma \, u^{(k)} \, - \,
\bC(U) \Gamma_0^{-1}u^{(j)} +\,\sqrt{2\bC(U)} \, \dot{\textbf{W}}^{(j)}.
\end{align}
This dynamical system is similar to the noisy EKI \eqref{eq:eki_noise}
but has a different noise structure (noise in parameter space not data space)
and explicitly accounts for the prior on the right hand side (rather
than having it enter through initialization). Inclusion of the
Tikhonov regularization term within EKI is introduced and studied
in \citep{chada2019tikhonov}.

Note that in the linear case \eqref{eq:linear}
the two systems \eqref{eq:implement-add} and \eqref{eq:implement}  are identical.  It is also natural to conjecture that if the particles are close
to one another then \eqref{eq:implement-add} and \eqref{eq:implement} will generate similar particle distributions. Based on this exact (in the linear case)
and conjectured (in the nonlinear case) relationship we
propose \eqref{eq:implement} as a derivative-free algorithm to approximately sample the Bayesian posterior distribution, and we propose
\eqref{eq:implement-add} as a natural object of analysis in order to understand this sampling algorithm.

\subsection{Mean Field Limit}
In order to write down the mean field limit of \eqref{eq:implement-add}, we define the macroscopic mean and covariance:
\begin{align*}
   m(\rho):=\int v\rho\,\rd v  \,,\qquad
   \CC(\rho) := \int  \bigl(v - m(\rho)\bigr) \otimes \bigl(v - m(\rho)\bigr) \,\rho(v) \, \rd v\,.
\end{align*}
Taking the large  particle limit leads to the mean field  equation
\begin{align} \label{eq:mf-geki}
\dot{u} = -\, \CC(\rho) \nabla\Phi_R(u) + \sqrt{2 \, \CC(\rho)} \, \dot{W},
\end{align}
with corresponding nonlinear Fokker-Planck equation
\begin{align}
\label{eq:NLFP}
\partial_t \rho = \nabla \cdot \bigl( \rho \, \CC(\rho) \nabla\Phi_R(u)\bigr) +  \, \CC(\rho): D^2 \rho\,.
\end{align}
Here $A_1:A_2$ denotes the Frobenius inner-product between matrices $A_1$ and $A_2$.
The existence and form of the mean-field limit is suggested by the
exchangeability of the process (existence) and by application of the
law of large numbers (form). Exchangeability is exploited in a related
context in \citep{del2017stability,del2018stability}.
The rigorous derivation of the mean-field equations \eqref{eq:mf-geki} and \eqref{eq:NLFP} is left for future work; for foundational work
 relating to mean field limits, see \citep{sznitman,Jabin2017,CFTV10,HaTadmor08,PareschiToscani_book,Toscani06Opinion} and the references therein.
The following lemma states the intuitive fact that the covariance, which plays a central role in equation
\eqref{eq:NLFP}, vanishes only for Dirac measures.

\begin{lemma}\label{lem:zerocov}
The only probability densities $\rho\in \mP(\R^d)$ at which $\CC(\rho)$ vanishes are Diracs,
$$
\rho(u)=\delta_{v}(u) \text{ for some } v\in\R^d \quad \Leftrightarrow \quad \CC(\rho)=0\,.
$$
\end{lemma}

\begin{proof}
That $\CC(\delta_v)=0$ follows by direct substitution. For the converse, note that $\CC(\rho)=0$ implies $\int |u|^2\rho\, \rd u=\left(\int u\rho\, \rd u\right)^2$, which is the equality case of Jensen's inequality, and therefore only holds if $\rho$ is the law of a constant random variable.
\end{proof}

\nc

%
%
%
%

\section{Theoretical Properties}
\label{sec:TP}

In this section we discuss theoretical properties of \eqref{eq:NLFP} which motivate the use of \eqref{eq:implement-add} and
\eqref{eq:implement} as particle systems to generate approximate samples from
the posterior distribution \eqref{eq:post2}. In Subsection \ref{ssec:NP} we exhibit a gradient
flow structure for \eqref{eq:NLFP} which shows that solutions evolve towards the posterior
distribution \eqref{eq:post2} unless they collapse to a Dirac measure. In Subsection \ref{ssec:LP}
we show that in the linear case, collapse to a Dirac does not occur if the initial condition is a Gaussian with non-zero covariance, and instead convergence to the
posterior distribution is obtained. In Subsection \ref{ssec:KWGF} we introduce a novel metric structure which
underpins the results in the two preceding sections, and will allow for a rigorous analysis of the long-term
behavior of the nonlinear Fokker-Planck equation in future work.

\subsection{Nonlinear Problem}
\label{ssec:NP}
Because $\CC(\rho)$ is independent of $u$, we may write equation \eqref{eq:NLFP} in divergence form, which facilitates the revelation of a gradient structure:
\begin{align}
\label{eq:NLFP2}
\partial_t \rho = \nabla \cdot \bigl( \rho \, \CC(\rho) \nabla\Phi_R(u) +  \, \rho \, \CC(\rho) \nabla\ln \rho\bigr)\,,
\end{align}
where we use the fact $\rho\nabla\ln\rho=\nabla\rho$.
Indeed, equation \eqref{eq:NLFP2} is nothing but the Fokker-Planck equation for \eqref{eq:eki-basic} for a time-dependent matrix $\bC(t)=\CC(\rho)$.
Thanks to the divergence form, it follows that \eqref{eq:NLFP2} conserves mass along the flow, and so we may assume $\int\rho(t,u)\,\rd u =1$ for all $t\ge0$. Defining the energy
\begin{align}
\label{eq:energy}
E(\rho)=\int \Bigl(\rho(u)\Phi_R(u)+\rho(u)\ln \rho(u)\Bigr)\,\rd u\,  ,
\end{align}
solutions to \eqref{eq:NLFP2} can be written as a gradient flow:
\begin{align}
\label{eq:NLFP3}
\partial_t \rho = \nabla \cdot \left( \rho \, \CC(\rho) \nabla \frac{\delta E}{\delta \rho} \right)\,,
\end{align}
where $\frac{\delta}{\delta\rho}$ denotes the $L^2$ first variation.
This will be made more explicit in Section~\ref{ssec:KWGF}, see Proposition \ref{prop2}.
Thanks to the gradient flow structure \eqref{eq:NLFP3}, stationary states of \eqref{eq:NLFP} are given either by critical points of the energy $E$, or by choices of $\rho$ such that $\CC(\rho)=0$ as characterized in Lemma~\ref{lem:zerocov}.
Critical points of $E$ solve the corresponding Euler-Lagrange condition
\begin{equation}\label{eq:EL}
    \frac{\delta E}{\delta \rho}= \Phi_R(u)+\ln\rho(u)=c \qquad \text{ on } \supp(\rho)
\end{equation}
for some constant $c$. The unique solution to \eqref{eq:EL} with unit mass is given by the Gibbs measure
\begin{equation}\label{eq:Gibbs}
    \rho_\infty(u):=\frac{e^{-\Phi_R(u)}}{\int e^{-\Phi_R(u)}\,\rd u}\,.
\end{equation}
Then, up to an additive normalization constant, the energy $E(\rho)$ is exactly the relative entropy of $\rho$ with respect to $\rho_\infty$, also known as the Kullback-Leibler divergence $\KL(\rho(t)\|\rho_\infty)$,
\begin{align*}
  E(\rho)
  &= \int \left(\Phi_R+\ln\rho(t)\right)\rho\,\rd u \\
  &=\int \frac{\rho(t)}{\rho_\infty} \ln\left(\frac{\rho(t)}{\rho_\infty}\right)\, \rho_\infty\,\rd u
  +\ln\left(\int e^{-\Phi_R(u)}\,\rd u\right)\\
   &=\KL(\rho(t)\|\rho_\infty)
  +\ln\left(\int e^{-\Phi_R(u)}\,\rd u\right)\,.
\end{align*}
Thanks to the gradient flow structure \eqref{eq:NLFP3}, we can compute the dissipation of the energy
\begin{equation}\label{fisher}
\begin{split}
    \frac{\rd}{\rd t}\Bigl\{ E(\rho)\Bigr \}&=\left\langle \frac{\delta E}{\delta \rho}\,,\, \partial_t \rho  \right\rangle_{L^2(\R^d)}\\
    &=-\int \rho \left\langle \nabla \frac{\delta E}{\delta \rho}, \CC(\rho) \nabla \frac{\delta E}{\delta \rho} \right\rangle\,\rd u\\
&=-\int \rho \, \Bigl| \CC(\rho)^{\frac12} \nabla (\Phi_R+\ln \rho)\Bigr|^2\,\rd u\,.
\end{split}
\end{equation}
As a consequence,  the energy $E$ decreases along trajectories until either $\CC(\rho)$ approaches zero (collapse to a Dirac
measure by Lemma~\ref{lem:zerocov}) or $\rho$ becomes the Gibbs measure with density $\rho_\infty$.

The dissipation of the energy along the evolution of the classical Fokker-Planck equation is known as the Fisher information \citep{vil2008}. We reformulate equation \eqref{fisher} by defining the following generalized Fisher information for any covariance matrix $\Lambda$,
\begin{equation*}
\mathcal{I}_{\Lambda}(\rho(t)\|\rho_\infty):=\int \rho\,\left\langle \nabla \ln \left(\frac{\rho}{\rho_\infty}\right)\,,\, \Lambda\nabla \ln \left(\frac{\rho}{\rho_\infty}\right)\right\rangle\,\rd u\,.
\end{equation*}
One may also refer to $\mI_\Lambda$ as a Dirichlet form as it is known in the theory of large particle systems, since we can write
\begin{equation*}
\mathcal{I}_{\Lambda}(\rho(t)\|\rho_\infty)=4\int \rho_\infty\,\left\langle \nabla \sqrt{\frac{\rho}{\rho_\infty}}\,,\, \Lambda \nabla \sqrt{\frac{\rho}{\rho_\infty}}\right\rangle\,\rd u\,.
\end{equation*}
For $\Lambda=\CC(\rho)$, we name functional $\mathcal{I}_{\mathcal{C}}$ the relative {\em Kalman-Fisher information}.
We conclude that the following energy dissipation equality holds,
\begin{equation*}
    \frac{\rd}{\rd t} \KL(\rho(t)\|\rho_\infty) = -\mathcal{I}_{\mathcal{C}}(\rho(t)\|\rho_\infty)\,.
\end{equation*}
To derive a rate of decay to equilibrium in entropy, we aim to identify conditions on $\Phi_R$ such that the following logarithmic Sobolev inequality holds: there exists $\lambda>0$ such that
\begin{equation}\label{logSob}
 \KL(\rho(t)\|\rho_\infty)\leq \frac{1}{2\lambda} \mI_{I_d}\left(\rho(t)\|\rho_\infty\right) \qquad \forall \rho\,.
\end{equation}
By \citep{BakryEmery}, it is enough to impose sufficient convexity on $\Phi_R$, i.e.
$D^2\Phi_R \ge \lambda I_d$, where $D^2\Phi_R$ denotes the Hessian of $\Phi_R$. This allows us to deduce convergence to equilibrium as long as $\CC(\rho)$ is uniformly bounded from below following standard arguments for the classical Fokker-Planck equation as presented for example in \citep{MarkoVillani}.

\begin{proposition}\label{prop:entropydecay}
Assume there exists $\alpha>0$ and $\lambda>0$ such that
\begin{equation*}
    \CC(\rho(t)) \ge \alpha I_d\,,\qquad
       D^2\Phi_R \ge \lambda I_d\,.
\end{equation*}
Then any solution $\rho(t)$ to \eqref{eq:NLFP2} with initial condition $\rho_0$ satisfying $\KL(\rho_0\|\rho_\infty)<\infty$ decays exponentially fast to equilibrium: there exists a constant $c=c(\rho_0,\Phi_R)>0$ such that for any $t>0$,
\begin{equation*}
    \|\rho(t)-\rho_\infty\|_{L^1(\R^d)} \leq  c e^{-\alpha \lambda t}\,.
\end{equation*}
\end{proposition}

This rate of convergence can most likely be improved using the correct logarithmic Sobolev inequality weighted by the covariance matrix $\CC$. However, the above estimate already indicates the effect of having the covariance matrix $\CC$ present in the Fokker-Planck equation \eqref{eq:NLFP2}. The properties of such inequalities in a more general setting is an interesting future avenue to explore. The weighted logarithmic Sobolev inequality that is well adapted to the setting here depends on the geometric structure of the Kalman-Wasserstein metric, see related studies in \citep{LiG}.

\begin{proof}
Thanks to the assumptions, and using the logarithmic Sobolev inequality \eqref{logSob}, we obtain decay in entropy,
\begin{align*}
\frac{\rd}{\rd t} \KL(\rho(t)\|\rho_\infty)
&\le - \alpha \mI_{I_d}(\rho(t)|\rho_\infty)
\le -2\alpha\lambda \KL(\rho(t)\|\rho_\infty) \,.
\end{align*}
We conclude using the Csisz\'ar-Kullback inequality as it is mainly known to analysts, also referred to as Pinsker inequality in probability (see \citep{AMTU} for more details):
\begin{equation*}
   \frac{1}{2} \|\rho(t)-\rho_\infty\|^2_{L^1(\R^d)}
   \le \KL(\rho(t)\|\rho_\infty)
   \le \KL(\rho_0\|\rho_\infty)e^{-2\alpha\lambda t}\,.
\end{equation*}
\end{proof}

\subsection{Linear Problem}
\label{ssec:LP}

Here we show that, in the case of a linear forward operator $\G$, the
Fokker-Planck equation (which is still nonlinear) has exact Gaussian solutions.
This property may be seen to hold in two ways: (i) by considering the
case in which the covariance matrix is an exogenously defined function of time alone, in which case the observation is straightforward; and (ii)  because the
mean field equation \eqref{eq:mf-geki} leads to exact closed equations for the mean and covariance. Once
the covariance is known the nonlinear Fokker-Planck equation \eqref{eq:NLFP} becomes linear,
and is explicitly solvable if $\G$ is linear and the initial condition is Gaussian.
Consider equation \eqref{eq:mf-geki} in the context of a linear observation map \eqref{eq:linear}.
The misfit is given by \eqref{eq:phi-reg}, and the gradient of $\Phi_R$
is given in \eqref{eq:gather}.
Note that since we assume that the covariance matrix $\Gamma_0$ is invertible, it is then also strictly positive-definite. Thus it follows that $B$ is strictly positive-definite and hence invertible too. We define $u_0:=B r$ noting that this is the solution of the regularized normal equations defining the minimizer of $\Phi_R$ in this
linear case; equivalently $u_0$ maximizes the posterior density. Indeed by completing the square we see that we may write
\begin{equation}
    \label{eq:infty}
    \rho_{\infty}(u) \propto \exp\Bigl(-\frac12\|u-u_0\|_{B}^2\Bigr).
\end{equation}

\begin{lemma}
\label{eq:roil}
Let $\rho(t)$ be a solution of \eqref{eq:NLFP} with $\Phi_R(\cdot)$ given by \eqref{eq:phi-reg}. Then the mean $m(\rho)$ and covariance matrix $\CC(\rho)$ are determined by $\mm(t)$ and $\mC(t)$ which satisfy the evolution equations
\begin{subequations}
\begin{align}
    \frac{\rd}{\rd t} \mm(t) &= - \mC(t) (B^{-1}\mm(t) - r) \label{eq:mudot}\\
    \frac{\rd}{\rd t} \mC(t)&=-2\mC(t)B^{-1}\mC(t) + 2\mC(t) \label{eq:Cdot}.
\end{align}
\end{subequations}
In addition, for any $\mC(t)$ satisfying \eqref{eq:Cdot}, its determinant and inverse solve
\begin{align}
    \frac{\rd}{\rd t} \det\mC(t)&=-2\left(\det\mC(t)\right)\Tr\left[B^{-1}\mC(t)-I_d\right]\,, \label{eq:Cdetdot}
\end{align}
\begin{equation}
\label{eq:display}
\frac{\rd}{\rd t}\bigl( \mC(t)^{-1}\bigr)=2B^{-1}-2\mC(t)^{-1}.
\end{equation}
As a consequence $\mC(t) \to B$ and $\mm(t) \to u_0$ exponentially as $t \to \infty$.
\end{lemma}
In fact, solving the ODE \eqref{eq:display} explicitly and using \eqref{eq:mudot}, exponential decay immediately follows:
\begin{align}
   & \mC(t)^{-1}=\left(\mC(0)^{-1}-B^{-1}\right)e^{-2t}+B^{-1}\,,\label{eq:Cinversesol}
   \end{align}
   and
   \begin{align}
    &\|\mm(t)-u_0\|_{\mC(t)}=\|\mm(0)-u_0\|_{\mC(0)}e^{-t}\,.\label{eq:msol}
\end{align}
\begin{proof}

We begin by deriving the evolution of the first and second moments. This is most easily accomplished by working
with the mean-field flow SDE \eqref{eq:mf-geki}, using the regularized
linear misfit written in \eqref{eq:phi-reg}. This yields the update
\begin{align*}
    \dot u = - \CC(\rho) \, (B^{-1}u - r) + \sqrt{2 \, \CC(\rho)} \, \dot{\textbf{W}}\,,
\end{align*}
where $\dot{\textbf{W}}$ denotes a zero mean random variable. Identical results can be obtained by working directly with
the PDE for the density, namely \eqref{eq:NLFP} with the regularized
linear misfit given in \eqref{eq:phi-reg}.
Taking expectations with
respect to $\rho$ results in
\begin{align*}
    \dot m(\rho) = - \CC(\rho) \, (B^{-1}m(\rho) - r).
\end{align*}
Let us use the following auxiliary variable $e = u - m(\rho)$. By linearity of
differentiation we can write
\begin{align*}
    \dot e = -  \CC(\rho) \, B^{-1} \, e + \sqrt{2\, \CC(\rho)} \, \dot{\textbf{W}}.
\end{align*}
By definition of the covariance operator, $\CC(\rho) = \E[e \otimes e]$, its
derivative with respect to time can be written as
\begin{align*}
    \dot \CC (\rho) = \E[\dot e \otimes e + e \otimes \dot e].
\end{align*}
However we must also include the It\^o correction, using It\^o's formula, and we can write the evolution equation of the covariance
operator as
\begin{align*}
    \dot \CC (\rho) = - 2 \, \CC(\rho)\,B^{-1}\,\CC(\rho) + 2\, \CC(\rho).
\end{align*}
This concludes the proof of \eqref{eq:Cdot}. For the evolution of the determinant and inverse, note that
\begin{align*}
     \frac{\rd}{\rd t} \det\CC(\rho)
     &=\Tr\left[\det\CC(\rho) \, \CC(\rho)^{-1} \,  \frac{\rd}{\rd t}\CC(\rho)\right]\,,
     \quad
      \frac{\rd}{\rd t} \CC(\rho)^{-1} = -\CC(\rho)^{-1} \left( \frac{\rd}{\rd t} \CC(\rho)\right) \CC(\rho)^{-1}\,,
\end{align*}
and so \eqref{eq:Cdetdot}, \eqref{eq:display} directly follow. Finally, exponential decay is a consequence of the explicit expressions \eqref{eq:Cinversesol} and \eqref{eq:msol}.
\end{proof}

Thanks to the evolution of the covariance matrix and its determinant, we can deduce that there is a family of Gaussian initial conditions that stay Gaussian along the flow and converge to the equilibrium $\rho_\infty$.

\begin{proposition}
\label{p:p}
Fix a vector $m_0\in \R^d$, a matrix $\CC_0\in\R^{d\times d}$ and take as
initial density the Gaussian distribution
\begin{equation*}
    \rho_0(u):=\frac{1}{(2\pi)^{d/2}}(\det \CC_0)^{-1/2} \exp\left(-\frac{1}{2} ||u-m_0||^2_{\CC_0}\right)
\end{equation*}
with mean $m_0$ and covariance $\CC_0$. Then the Gaussian profile
\begin{equation*}
        \rho(t,u):=\frac{1}{(2\pi)^{d/2}}(\det \mC(t))^{-1/2} \exp\left(-\frac{1}{2}\bigl|\bigl| u-\mm(t)\bigl|\bigl|_{\mC(t)}^2\right)
\end{equation*}
solves evolution equation \eqref{eq:NLFP} with initial condition $\rho(0,u)=\rho_0(u)$, and where $\mm(t)$ and $\mC(t)$ evolve
according to \eqref{eq:mudot} and \eqref{eq:Cdot} with initial conditions $m_0$ and $\CC_0$. As a consequence, for such initial
conditions $\rho_0(u)$, the solution of the Fokker-Planck equation \eqref{eq:NLFP} converges to $\rho_{\infty}(u)$ given by \eqref{eq:infty} as $t \to \infty.$
\end{proposition}
\begin{proof}
It is straightforward to see that, for $m(\rho)$ and $\CC(\rho)$ given by Lemma \ref{eq:roil},
$$\nabla \rho=-\CC(\rho)^{-1}(u-m(\rho))\,\rho,$$
since both $m(\rho)$ and $\CC(\rho)$ are independent of $u$. Therefore,
substituting the Gaussian ansatz $\rho(t,u)$ into the first term in the right
hand side of \eqref{eq:NLFP}, we have
\begin{align}
    \nabla \cdot\left(\rho \, \CC(\rho)(B^{-1}u - r)\right)
    &= (\nabla\rho )\cdot \CC(\rho)(B^{-1}u-r) + \rho \nabla \cdot(\CC(\rho)B^{-1}u) \notag\\
    &= \left(-\CC(\rho)^{-1}(u-m(\rho))\cdot \CC(\rho) (B^{-1}u-r) +  \Tr[\CC(\rho)B^{-1}]\right)\rho\notag\\
    &= \left(- \bigl|\bigl|u-m(\rho)\bigl|\bigl|^2_{B} +  \Bigl\langle u - m(\rho), u_0 - m(\rho)\Bigl\rangle_{B} + \Tr[\CC(\rho)B^{-1}]\right)\rho,
\end{align}
where $B^{-1} = A^\top\Gamma^{-1} A + \Gamma_0^{-1}$, $r = A^\top\Gamma^{-1} y$ and $u_0 = B\,r$.  Recall that $B^{-1}$ is invertible.
The second term on the right hand side of \eqref{eq:NLFP} can be simplified, as follows
\begin{align}
\CC(\rho): D^2 \rho &= \CC(\rho) : \Bigl( - \CC(\rho)^{-1} + \bigl(\CC(\rho)^{-1} (u - m(\rho))\bigr) \otimes \bigl(\CC(\rho)^{-1} (u - m(\rho))\bigr) \Bigr) \rho \notag \\
&= \left(- \Tr[I_d] + ||u - m(\rho)||^2_{\CC(\rho)}\right) \rho.
\end{align}
Thus, combining the previous two equations, the right hand side of \eqref{eq:NLFP} is given by  the following expression
\begin{align}
\Biggl[\Tr[B^{-1}\CC(\rho) - I_{d}] - ||u-m(\rho)||^2_{B}+ \bigl|\bigl|u-m(\rho)\bigl|\bigl|^2_{\CC(\rho)} + \Bigl\langle u - m(\rho), u_0 - m(\rho)\Bigl\rangle_{B}  \Biggl] \rho.
\end{align}
For the left-hand side of \eqref{eq:NLFP}, note that by \eqref{eq:mudot} and \eqref{eq:Cdot},
\begin{align*}
\frac{\rd}{\rd t} \bigl|\bigl| u-m(\rho)\bigl|\bigl|_{\CC(\rho)}^2
&= 2\left\langle \frac{\rd}{\rd t} (u-m(\rho))\,,\,\CC(\rho)^{-1}(u-m(\rho))\right\rangle \\
& \quad + \left\langle (u-m(\rho))\,,\,\frac{\rd}{\rd t}\bigl(\CC(\rho)^{-1}\bigr)(u-m(\rho))\right\rangle\\
&=- 2\bigl\langle u-m(\rho), u_0 - m(\rho)\bigl\rangle_{B} +2\, ||u-m(\rho)||^2_{B} - 2\, ||u-m(\rho)||^2_{\CC(\rho)}
\end{align*}
and therefore, combining with \eqref{eq:Cdetdot},
\begin{align}
\partial_t\rho
&= \left[-\frac{1}{2}(\det\CC(\rho))^{-1}\left(\frac{\rd}{\rd t}\det\CC(\rho)\right)
-\frac{1}{2}\frac{\rd}{\rd t} || u-u_0||_{\CC(\rho)}^2\right] \rho \notag\\
&= \Biggl[\Tr[B^{-1}\CC(\rho) - I_{d}] - ||u-m(\rho)||^2_{B}+ \bigl|\bigl|u-m(\rho)\bigl|\bigl|^2_{\CC(\rho)} + \Bigl\langle u - m(\rho), u_0 - m(\rho)\Bigl\rangle_{B}  \Biggl] \rho,
\end{align}
which concludes the first part of the proof. The second part concerning the large time asymptotics is a straightforward consequence of the
asymptotic behaviour of $\mm$ and $\mC$ detailed in Lemma \ref{eq:roil}.
\end{proof}

In the case of the classical Fokker-Planck equation $\mC(t)=I_d$ with a quadratic confining potential, the result in Proposition \ref{p:p} follows from the fact that the fundamental solution of \eqref{eq:NLFP} is a Gaussian, see  \citep{CarrilloToscani1998}.

\begin{corollary}
\label{p:p2}
Let $\rho_0$ be a non-Gaussian initial condition for \eqref{eq:NLFP} in the case where $\Phi_R$ is given by \eqref{eq:phi-reg}.
Assume that $\rho_0$ satisfies $\KL(\rho_0\|\rho_\infty)<\infty$. Then any solution of \eqref{eq:NLFP} converges exponentially fast to $\rho_{\infty}$
given by \eqref{eq:Gibbs} as $t \to \infty$ both in entropy, and in $L^1(\R^d)$.
\end{corollary}

\begin{proof}
Let $a\in\R^d$ have Euclidean norm $1$ and define $q(t):=\langle a, \mC(t)^{-1} a \rangle$. From equation \eqref{eq:display} it follows that
$$\dot{q} \le 2\lambda-2q$$
where $\lambda$ is the maximum eigenvalue of $B^{-1}$.
Hence it follows that $q$ is bounded above, independently of $a$, and that hence $\mC$ is bounded from below as an operator. Together with the fact that the Hessian $D^2\Phi_R=B^{-1}$ is bounded from below, we conclude using Proposition~\ref{prop:entropydecay}.
\end{proof}

\subsection{Kalman-Wasserstein Gradient Flow}
\label{ssec:KWGF}

We introduce an infinite-dimensional Riemannian metric structure, which we name the Kalman-Wasserstein metric, in density space.
It allows the interpretation of solutions to equation~\eqref{eq:NLFP} as gradient flows in density space. 
To this end we denote by $\mathcal{P}$ the space of probability measures on a convex set $\Omega\subseteq \R^d$:
\begin{equation*}
    \mathcal{P}:=\left\{\rho\in L^1(\Omega)\,:\, \rho\ge 0 \text{ a.e. }\,,\, \int \,\rho(x)\,\rd x =1\right\}\,.
\end{equation*}
The probability simplex $\mathcal{P}$ is a manifold with boundary. For simplicity, we focus on the subset
$$
\mP_+:= \left\{\rho\in \mathcal{P}\, :\, \rho>0\text{ a.e. }\,,\, \rho\in C^\infty(\Omega) \right\}\,.
$$
The tangent space of $\mP_+$ at a point $\rho\in\mP_+$ is given by
\begin{align*}
    \Tp&=\left\{ \left.\frac{\rd}{\rd t}\rho(t)\right|_{t=0}\,:\, \rho(t) \text{ is a curve in } \mP_+\,,\, \rho(0)=\rho \right\}\\
    &= \left\{ \sigma\in C^{\infty}(\Omega)\colon \int \sigma dx=0 \right\}\,.
\end{align*}
The second equality follows since for all $\sigma\in \Tp$ we have $\int \sigma(x)\, \rd x=0$ as the mass along all curves in $\mathcal{P}_+$ remains constant. For the set $\mP_+$, the tangent space $\Tp$ is therefore independent of the point $\rho\in\mP_+$.
 Cotangent vectors are elements of the topological dual $\cTp$ and can be identified with tangent vectors via the action of the \emph{Onsager operator} \citep{Mielke16Ons,Ons31p1,Ons31p2,OnM53,Ott05}
$$\Ons:\cTp\to\Tp\,.$$

In this paper, we introduce the following new choice of Onsager operator:
\begin{equation}\label{def-Ons}
    \Ons(\phi)=-\nabla\cdot\left(\rho\CC(\rho)\nabla \phi\right)=: \left(-\Onsop\right)\phi\,.
\end{equation}
By Lemma~\ref{lem:zerocov}, the weighted elliptic operator $\Onsop$ becomes degenerate if $\rho$ is a Dirac. For points $\rho$ in the set $\mathcal{P}_+$ that are bounded away from zero, the operator $\Onsop$ is well-defined, non-singular and invertible since $\rho\CC(\rho)>0$. So we can write
\begin{align*}
    \Ons^{-1}\,:\,&\Tp\to\cTp,\quad\sigma\mapsto \left(-\Onsop\right)^{-1} \sigma\,.
\end{align*}
This provides a 1-to-1 correspondence between elements $\phi\in\cTp$ and $\sigma\in\Tp$. For general $\rho\in\mP_+$, we can instead use the pseudo-inverse $\left(-\Onsop\right)^{\dagger}$, see \citep{LiG}.
With the above choice of Onsager operator, we can define a generalized Wasserstein metric tensor:
\begin{definition}[Kalman-Wasserstein metric tensor]
Define
$$g_{\rho,\CC}\,:\, \Tp\times \Tp \to\R$$ as follows:
\begin{equation*}
    g_{\rho,\CC}(\sigma_1,\sigma_2)=\int_\Omega \left\langle\nabla \phi_1\,,\,\CC(\rho)\nabla \phi_2\right\rangle \,\rho\, \rd x,
\end{equation*}
where $\sigma_i=\left(-\Onsop\right)\phi_i=-\nabla\cdot\left(\rho\CC(\rho)\nabla \phi_i\right)\in T_\rho\mathcal{P}_+$ for $i=1,2$.
\end{definition}
With this metric tensor, the Kalman-Wasserstein metric $\mathcal{W}_{\mathcal{C}}\colon \mathcal{P}_+\times \mathcal{P}_+\rightarrow \mathbb{R}$ can be represented by the geometric action function. Given two densities $\rho^0$, $\rho^1\in\mathcal{P}_+$, consider
\begin{align*}
W_{\mathcal{C}}(\rho^0, \rho^1)^2=\inf&\int_0^1 \int_{\Omega} \left\langle\nabla \phi_t\,,\,\CC(\rho_t)\nabla \phi_t\right\rangle \,\rho_t\, \rd x\\
\text{ subject to}&\quad\partial_t\rho_t+\nabla\cdot(\rho_t\CC(\rho_t)\nabla\phi_t)=0,~\rho_0=\rho^0,~\rho_1=\rho^1,
\end{align*}
where the infimum is taken among all continuous density paths $\rho_t:=\rho(t,x)$ and potential functions $\phi_t:=\phi(t,x)$.
The Kalman-Wasserstein metric has several interesting mathematical properties, which will be the focus of future work. In this paper, working in $(\mathcal{P}_+, g_{\rho,\mathcal{C}})$, we derive the gradient flow formulation that underpins the formal calculations given in Subsection~\ref{ssec:NP} for the energy functional $E$ defined in \eqref{eq:energy}.
\begin{proposition}\label{prop2}
Given a finite functional $\mF:\mathcal{P}_+\to\R$, the gradient flow of $\mathcal{F}(\rho)$ in $(\mathcal{P}_+, g_{\rho,\mathcal{C}})$ satisfies
\begin{equation*}
 \partial_t\rho=\nabla\cdot\left(\rho\,\CC(\rho)\nabla \frac{\delta \mF}{\delta \rho}\right)\,.
\end{equation*}
\end{proposition}
\begin{proof}
The Riemannian gradient operator $\textrm{grad}\mathcal{F}(\rho)$ is defined via the metric tensor $g_{\rho, \mathcal{C}}$ as follows:
		\begin{equation*}
		g_{\rho, \mathcal{C}}(\sigma, \textrm{grad}\mathcal{F}(\rho))= \int_\Omega \frac{\delta}{\delta\rho(u)}\mathcal{F}(\rho) \sigma(u) \rd u\,,\qquad \forall \sigma\in T_\rho\mathcal{P}_+\,.
		\end{equation*}
		Thus, for $\phi:=(-\Delta_{\rho,\CC})^{-1}\sigma\in T_\rho^*\mathcal{P}_+$, we have
		\begin{equation*}
		\begin{split}
		  g_{\rho, \mathcal{C}}(\sigma, \textrm{grad}\mathcal{F}(\rho))
		  =& \int \phi(u) \textrm{grad}\mathcal{F}(\rho)\,\rd u
		  =-\int\nabla\cdot(\rho \mathcal{C}(\rho)\nabla\phi)\frac{\delta}{\delta\rho}\mathcal{F}(\rho)\,\rd u\\
		  =& \int \left\langle\nabla\phi\,,\, \CC(\rho)\nabla\frac{\delta}{\delta\rho}\mathcal{F}(\rho)\right\rangle\,\rho \,\rd u\\
		  =& -\int \phi(u) \nabla\cdot(\rho\mathcal{C}(\rho)\nabla\frac{\delta}{\delta\rho}\mathcal{F}(\rho))\,\rd u.
		\end{split}
		\end{equation*}
Hence
\begin{equation*}
\textrm{grad}\mathcal{F}(\rho)=-\nabla\cdot(\rho\mathcal{C}(\rho)\nabla\frac{\delta}{\delta\rho}\mathcal{F}(\rho)).
\end{equation*}
Thus we derive the gradient flow by
		\begin{equation*}
		  \partial_t\rho=-\textrm{grad}\mathcal{F}(\rho)=\nabla\cdot(\rho\mathcal{C}(\rho)\nabla\frac{\delta}{\delta\rho}\mathcal{F}(\rho)).
		\end{equation*}
\end{proof}

\begin{remark}
Our derivation concerns the gradient flow on the subset $\mP_+$ of $\mP$ for simplicity of exposition. However, a rigorous analysis of the evolution of the gradient flow \eqref{eq:NLFP3} requires to extend the above arguments to the full set of probabilities $\mP$, especially as we want to study Dirac measures in view of Lemma~\ref{lem:zerocov}.
If $\rho$ is an element of the boundary of $\mP$, one may consider instead the pseudo inverse of the operator $\Delta_{\rho,\mathcal{C}}$. This will be the focus of future work, also see the more general analysis in \citep{AmbrosioGigliSavare2005_gradienta}, e.g. Theorem 11.1.6.
\end{remark}

\section{Numerical Experiments}
\label{sec:NE}

In this section we demonstrate that the intuition developed in the previous two
sections does indeed translate into useful algorithms for generating approximate
posterior samples without computing derivatives of the forward map $\G$. We do
this by considering non-Gaussian inverse problems, defined through a nonlinear
forward operator $\G$, showing how numerical solutions of \eqref{eq:implement}
are distributed after large time, and comparing them with exact posterior
samples found from MCMC.

Achieving the mean-field limit requires $J$ large, and hence typically
larger than the dimension $d$ of the parameter space. There are
interesting and important problems arising in science and engineering in
which the number of parameters to be estimated is small,
even though evaluation of $\G$ involves solution of computationally expensive
PDEs; in this case choosing $J>d$ is not prohibitive. We also include
numerical results which probe outcomes when $J<d.$ To this end we study
two problems, the first an inverse problem for a two-dimensional vector arising
from a two point boundary-value problem, and the second an inverse problem
for permeability from pressure measurements in Darcy flow; in this second
problem the dimension of the parameter space is tunable from small up to
infinite dimension, in principle.

\subsection{Derivative-Free}\label{sec:implement}
In this subsection we describe how to use \eqref{eq:implement}
for the solution of the inverse problem \eqref{eq:IP}.
We approximate the continuous time stochastic
dynamics by means of a linearly implicit split-step
discretization scheme given by
\begin{subequations}\label{eq:implicit}
\begin{align}
    {u}^{(*, j)}_{n+1} &= {u}^{(j)}_{n} - \Delta t_n \, \frac{1}{J}\sum_{k = 1}^J \, \langle \G(u^{(k)}_n) - \bar{\G}, \G(u^{(j)}_n) - y \rangle_\Gamma \, u^{(k)}_n \, - \Delta t_n \, \bC (U_n) \, \Gamma_0^{-1} \, {u}^{(*, j)}_{n+1} \\
    {u}^{(j)}_{n+1} &= {u}^{(*, j)}_{n+1} +\,\sqrt{2 \, \Delta t_n\, \bC(U_n)} \,  \xi^{(j)}_n,
\end{align}
\end{subequations}
where $\xi^{(j)}_n \sim \N(0, I)$, $\Gamma_0$  is the prior covariance and $\Delta t_n$ is an adaptive timestep computed as in
\citep{KovachkiStuart2018_ensemble}.

\subsection{Gold Standard: MCMC}
In this subsection we describe the specific Random Walk Metropolis Hastings
(RWMH) algorithm used to solve the same Bayesian inverse problem as in the
previous subsection; we view the results as gold standard samples from the
desired posterior distribution.  The link between RWMH methods and Langevin
sampling is explained in the literature review within the introduction where
it is shown that the latter arises as a diffusion limit of the former,
as shown in numerous papers following on from the seminal work
in  \citep{roberts1997}.
The proposal distribution is a
Gaussian centered at the current state of the Markov chain with covariance given
by $\Sigma = \tau \times \bC(U^*)$, where $\bC(U^*)$ is the covariance computed
from the last iteration of the algorithm described in the preceding subsection,
and $\tau$ is a scaling factor tuned for an acceptance rate of approximately
 $25\%$ \citep{roberts1997}. In our case, $\tau = 4$.
The RWMH algorithm was used to get $N = 10^5$
samples with the Markov chain starting at an approximate solution given by the
mean of the last step of the algorithm from the previous subsection. For
the high dimensional problem we use the pCN variant on RWMH \citep{cotter2013mcmc}; this too has a diffusion limit of Langevin form \citep{pillai2014noisy}.

\subsection{Numerical Results: Low Dimensional Parameter Space}

The numerical experiment considered here is the example originally presented
in \citep{ernst2015analysis} and also used in \citep{herty2018kinetic}. We start by defining the forward map which is given by the one-dimensional elliptic boundary value problem
\begin{align}
    -\frac{\rd}{\rd x}\left(\exp(u_1) \, \frac{\rd}{\rd x} p(x) \right) = 1, \quad x \in [0,1],
\end{align}
with boundary conditions $p(0) = 0$ and $p(1) = u_2$.  The explicit solution for this problem, \cite[see][]{herty2018kinetic}, is given by
\begin{align}
    p(x) = u_2 x + \exp(-u_1) \left( -\frac{x^2}{2} + \frac{x}{2}\right).
\end{align}
The forward model operator $\G$ is then defined by
\begin{align}\label{eq:ebvp}
    \G(u) = \left( \begin{array}{c}
         p(x_1)\\
         p(x_2)
    \end{array}\right).
\end{align}
Here $u = (u_1, u_2)^\top$ is a constant vector that we want to find and we
assume that we are given noisy measurements $y$ of $p(\cdot)$ at locations $x_1
= 0.25$ and $x_2 = 0.75$. The precise Bayesian inverse problem considered here
is to find the distribution of the unknown $u$ conditioned on the observed data
$y$, assuming additive Gaussian noise $\eta \sim \N(0, \Gamma)$, where $\Gamma =
0.1^2 \, I_2$ and $I_2 \in \R^{2 \times 2}$ is the identity matrix. We use as
prior distribution $\N(0, \Gamma_0)$, $\Gamma_0 = \sigma^2 I_2$ with
$\sigma=10.$ The resulting Bayesian inverse problem is then solved,
approximately, by the algorithms we now outline and with with observed data $y =
(27.5, 79.7)^\top$.
Following \citep{herty2018kinetic}, we consider an initial ensemble drawn
from $\N(0, 1) \times \mathsf{U}(90, 110)$.

Figure~\ref{fig:results} shows the results for the solution of the Bayesian
inverse problem considered above. In addition to implementing the algorithms
described in the previous two subsections, we also employ a specific
implementation of the EKI formulation introduced in the paper of
\cite{herty2018kinetic}, and defined by the numerical discretization shown in
\eqref{eq:implicit}, but with $\bC(U)$ replaced by the identity matrix $I_2$;
this corresponds to the algorithm from equation (20) of
\cite{herty2018kinetic}, and in particular the last display of their Section 5,
with $\xi \sim \N(0,I_2).$ The blue dots correspond to the output of this
algorithm at the last iteration. The red dots correspond to the last ensemble of
the EKI algorithm as presented in \citep{KovachkiStuart2018_ensemble}. The
orange dots depict the RWMH gold standard described above. Finally, the green
dots shows the ensemble members at the last iteration of the proposed EKS
\eqref{eq:implement}. In this experiment, all versions of the ensemble Kalman
methods were run with the adaptive timestep scheme from Subsection
\ref{sec:implement} and all were run for $30$ iterations with an ensemble size
of $J = 10^3$.

\begin{figure}[!htp]
    \centering
    \includegraphics[width= .9\linewidth]{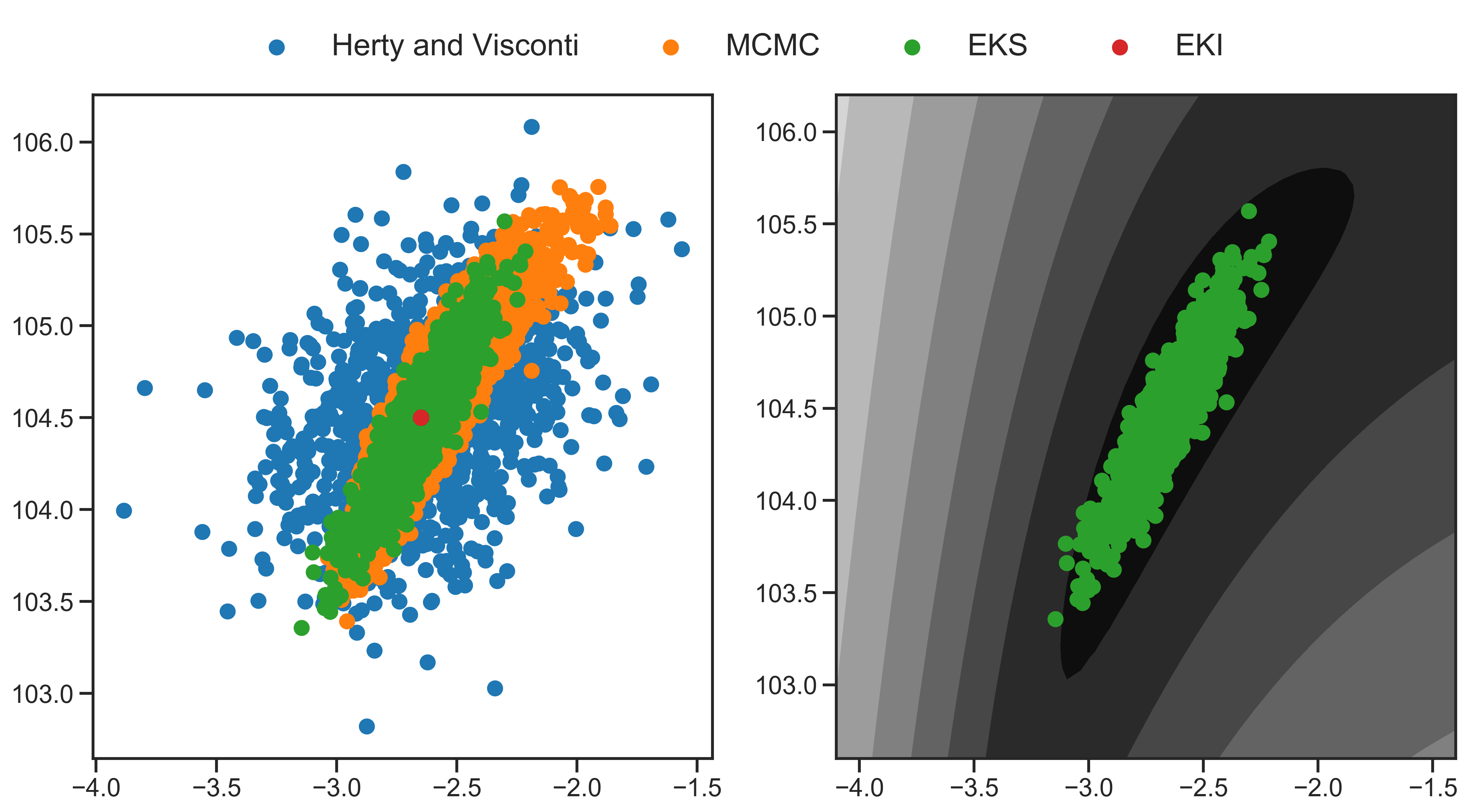}\\
    \vspace{1cm}
    \includegraphics[width= .8\linewidth]{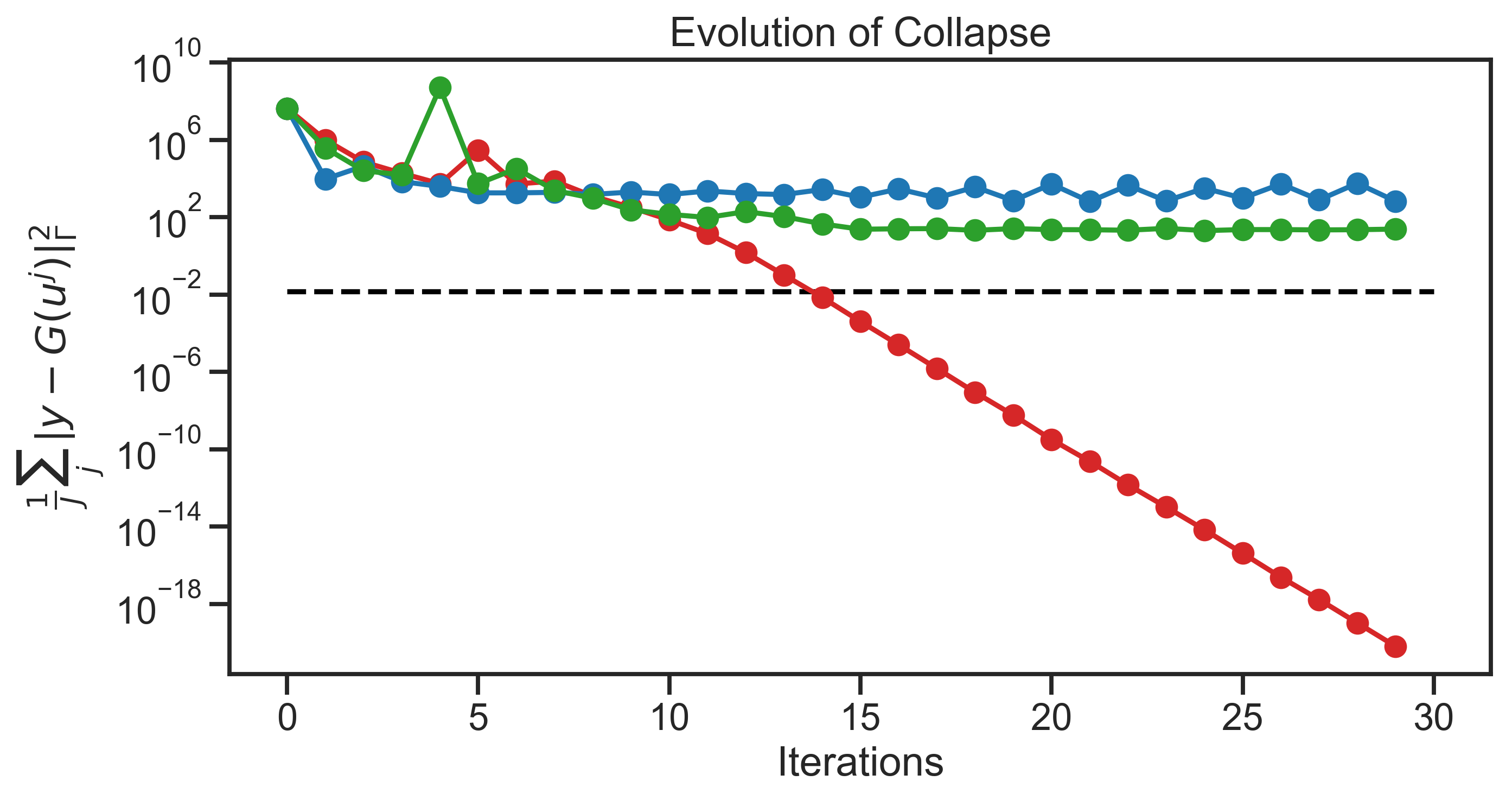}
    \caption{Results of applying different versions of ensemble Kalman methods to the non-linear elliptic boundary problem. For comparison, a Random Walk Metropolis Hastings
    algorithm is also displayed to provide a gold standard. The proposed EKS captures approximatel the true distribution, effectively avoiding overfitting or overdispersion shown with the other two implementations. Overfitting is clearly shown
    from the red line in the lower subfigure. The line in blue, shows overdispersion
    exhibited by the algorithm proposed in \citep{herty2018kinetic}. The upper right subfigure
    illustrates the approximation to the posterior. Color coding is consistent among the
    subfigures. }
    \label{fig:results}
\end{figure}

Consider first the top-left panel. The true distribution, computed by RWMH, is
shown in orange.  Note that the algorithm of \cite{KovachkiStuart2018_ensemble}
collapses to a point (shown in red), unable to escape overfitting, and relating
to a form of consensus formation. In contrast, the algorithm of
\cite{herty2018kinetic}, while avoiding overfitting, overestimates the spread of
the ensemble members, relative to the gold standard RWMH; this is exhibited by
the blue over-dispersed points. The proposed EKS (green points) gives results
close to the RWMH gold standard. These issues are further demonstrated in the
lower panel which shows the misfit (loss) function as a function of iterations
for the three algorithms (excluding RWMH); the red line demonstrates overfitting
as the misfit value falls below the noise level, whereas the other two
algorithms avoid overfitting.

We include the derivative-free optimization algorithm EKI (red points)
because it gives insight into what can be achieved with these ensemble based
methods in the absence of noise (namely derivative-free optimization);
we include the noisy EKI algorithm of \cite{herty2018kinetic} (blue points)
to demonstrate that considerable care is needed with the introduction of
noise if the goal is to produce posterior samples; and we include our
proposed EKS algorithm (green points) to demonstrate
that judicious addition of
noise to the EKI algorithm helps to produce approximate samples from the
true posterior distribution of the Bayesian inverse problem; we include
true posterior samples (orange points) for comparison. We reiterate that
the methods of \cite{chen2012ensemble,emerick2013investigation}
also hold the potential to
produce good approximate samples, though they suffer from the rigidity
of needing to be initialized at the prior and integrated to exactly
time $1$.

\subsection{Numerical Results: High Dimensional Parameter Space}

The forward problem of interest is to find the pressure field $p(\cdot)$ in a
porous medium defined by permeability field $a(\cdot)$; for simplicity we assume
that $a(\cdot)$ is a scalar-field in this paper. Given a scalar field $f$
defining sources and sinks of fluid, and assuming Dirichlet boundary conditions
on the pressure for simplicity, we obtain the following elliptic PDE for the
pressure:
\begin{subequations}
\begin{align}
-\nabla \cdot ( a(x) \nabla p(x)) &= f(x), \quad x \in D.\\
p(x) &=0, \qquad \, \, x \in \partial D.
\end{align}
\end{subequations}
In what follows we will work on the domain $D=[0,1]^2.$ We assume that the
permeability is dependent on unknown parameters $u \in \R^d$,
so that $a(x)=a(x;u).$ The inverse problem of interest is to determine $u$
from $d$ linear functionals (measurements) of $p(x;u)$, subject to additive
noise. Thus
\begin{equation}
\G_{j}(u)=\ell_j\bigl(p(\cdot;u)\bigr)+\eta_j, \quad j=1,\cdots, K.
\end{equation}

We will assume that $a(\cdot) \in L^{\infty}(D;\R)$ so that $p(\cdot) \in H^1_0(D;\R)$ and
thus we take the $\ell_j$ to be linear functionals on the space $H^1_0(D;\R)$.
In practice we will work with pointwise measurements so that $\ell_j(p)=p(x_j)$;
these are not elements of the dual space of $H^1_0(D;\R)$ in dimension $2$; but
mollifications of them are, and in practice mollification with a narrow kernel
does not affect results of the type presented here and so we do not use it
\citep{iglesias2015iterative}. We model $a(x;u)$ as a log-Gaussian field with
precision operator defined as
\begin{align}
  \mathcal{C}^{-1} = ( - \Delta + \tau^2 \mathcal{I})^\alpha,
\end{align}
where the Laplacian $\Delta $ is equipped with Neumann boundary conditions
on the space of spatial-mean zero functions,
and $\tau$ and $\alpha$ are known constants that control the
underlying lengthscales and smoothness of the
underlying random field. In our experiments $\tau = 3$, and $\alpha = 2.$
Such parametrization yields a Karhunen-Lo\`{e}ve (KL) expansion
\begin{equation}
\label{eq:adefine}
\log a(x; u)=\sum_{\ell \in K} u_\ell \, \sqrt{\lambda_\ell}\, \varphi_{\ell}(x)
\end{equation}
where the eigenpairs are of the form
\begin{align}
    \varphi_{\ell}(x)=\cos\bigl(\pi \langle \ell,x \rangle \bigr),  \qquad
    \lambda_\ell = (\pi^2 |\ell|^2 + \tau^2)^{-\alpha},
\end{align}
where $K \equiv \Z^2$ is the set of indices over which the random series
is summed  \nc and the $u_\ell \sim N(0,1)$ i.i.d.
\citep{pavliotis2014stochastic}.
In pratice we will approximate $K$ by $K_d \subset \Z^2$, a
set with finite cardinality $d$, and consider different $d.$
For visualization we will sometimes find it helpful to write
\eqref{eq:adefine} as a sum over a one-dimensional variable
rather than a lattice:
\begin{equation}
\label{eq:adefine2}
\log a(x; u)=\sum_{k \in \mathbb{Z}^+} u_k' \, \sqrt{\lambda_k'}\, \varphi_{k}'(x)
\end{equation}
We order the indices in $\mathbb{Z}^+$ so that the
eigenvalues $\lambda_k'$ are in descending order by size.

We generate a truth random field by constructing $u^\dagger
\in \R^d$ by sampling it from $\N(0, I_d)$, with $d = 2^8$ and $I_d$
the identity on $\R^d$ and using $u^\dagger$ as the coefficients in
\eqref{eq:adefine}. We create data $y$ from \eqref{eq:IP} with
$\eta \sim N(0,0.1^2 \times I_K).$ For the Bayesian inversion we choose
prior covariance $\Gamma_0=10^2 I_d;$
we also sample from this prior to initialize the ensemble for EKS.
We run the experiments with different
ensemble sizes to understand both strengths and limitations of  the proposed
algorithm for nonlinear forward models. Finally, we chose $J \in  \{8, 32, 128, 512, 2048\},$
which allows the study of both $J>d$ and $J<d$ within the methodology.

Results showing the solution of this Bayesian inverse problem by MCMC (orange
dots), with $10^5$ samples, and by the EKS with different $J$ (different shades
of green dots) are shown in \cref{fig:darcy_contrast} and
\cref{fig:darcy_identifiability}. For every ensemble size configuration, the EKS
algorithm was run until $2$ units of time were achieved. As can be seen from
\cref{sfig:darcy_negsob} the algorithm has reached an equilibrium after this
duration. The two dimensional
scatter plots in figure \ref{sfig:dracy_scatter} show components $u_{k}'$
with $k = 0, 1, 2$. That is, we are showing the components of $u$ which are
associated to the three largest eigenvalues in the KL expansion
\eqref{eq:adefine} under the posterior distribution. We can see that sample
spread is better matched to the gold standard MCMC spread as the size $J$ of the
EKS ensemble is increased. In \cref{sfig:darcy_negsob} and \cref{sfig:darcy_L2}
we show the evolution of the dispersion of the ensemble around its mean at every
time step, $\bar u(t)$, and around the truth $u^\dagger$. The metrics we use to test the ensemble spread
are
\begin{align}
  d_{H^{-2}}(\cdot) = \sqrt{\frac{1}{J} \sum_{j= 1}^J \|u^{(j)}(t) - \cdot \,  \|_{H^{-2}}^2}, \qquad
  d_{L^{2}}(\cdot) = \sqrt{\frac{1}{J} \sum_{j= 1}^J \|u^{(j)}(t) - \cdot \,  \|_{L^{2}}^2},
\end{align}
where both are evaluated at $\bar u(t)$ and $u^\dagger$ at every simulated time $t$.
For these metrics we use the norms defined by
\begin{align}
  \| u \|_{H^{-2}} = \sqrt{\sum_{\ell \in K_d} |u_\ell|^2 \lambda_\ell}, \qquad
  \| u \|_{L^2} = \sqrt{\sum_{\ell \in K_d} |u_\ell|^2},
\end{align}
where the first is defined in the negative Sobolev space $H^{-2}$, whilst the
second in the $L^2$ space. The first norm allows for higher discrepancy in the
estimation of the tail of the modes in equation \eqref{eq:adefine}. Whereas, the
second norm penalizes equally discrepancies in the tail of the KL expansion. In
\cref{sfig:darcy_negsob}, we see rapid convergence of the spread around the mean
and around the truth for all ensmeble sizes $J.$ The evolution in
\cref{sfig:darcy_negsob} for both cases shows that the algorithm reaches its
stationary distribution, while incorporating higher variability with increasing
ensemble size. The figures are similar because the posterior mean and the truth
are close to one another. Lower values of the metrics in
\cref{sfig:darcy_negsob} and \cref{sfig:darcy_L2} for smaller ensembles can be
understood due to a mixed effect of reduced variability and overfitting to the
MAP estimate of the Bayesian inverse problem. The results using the $L^2$ norm
in \cref{sfig:darcy_L2}, allows us to see more discrepancy between ensemble
sizes. Higher metric value for larger ensembles is due to the ensemble better
approximating the posterior, as will be discussed below. In summary,
\cref{fig:darcy_contrast} shows evidence that the EKS is generating samples from
a good approximation to the posterior and that this posterior is centred close
to the truth. Increasing the ensemble size improves these features of the EKS
method.
\begin{figure}
  \centering
  \subfigure[Bivariate scatter plots of the approximate posterior
distribution on the three largest modes (as ordered
by prior variance and here labelled $u_0, u_1, u_2$)
in the KL expansion \eqref{eq:adefine}. The pCN algorithm (orange dots) is used as a reference with $10^5$ samples.]{\includegraphics[width=\linewidth]{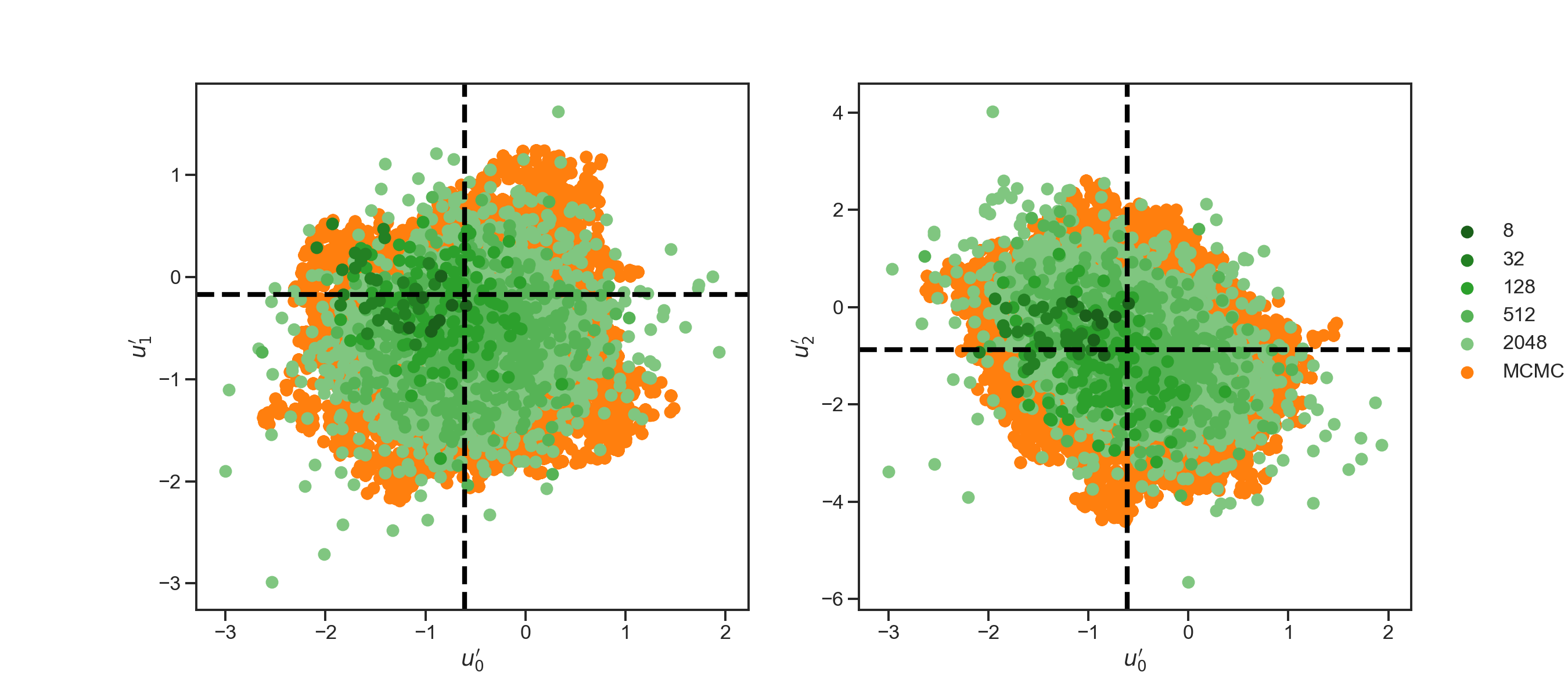} \label{sfig:dracy_scatter}}\\
  \subfigure[Evolution statistics of the EKS with respect to simulated time under the negative Sobolev norm $\|\cdot \|_{H^{-2}}$.]{\includegraphics[width=\linewidth]{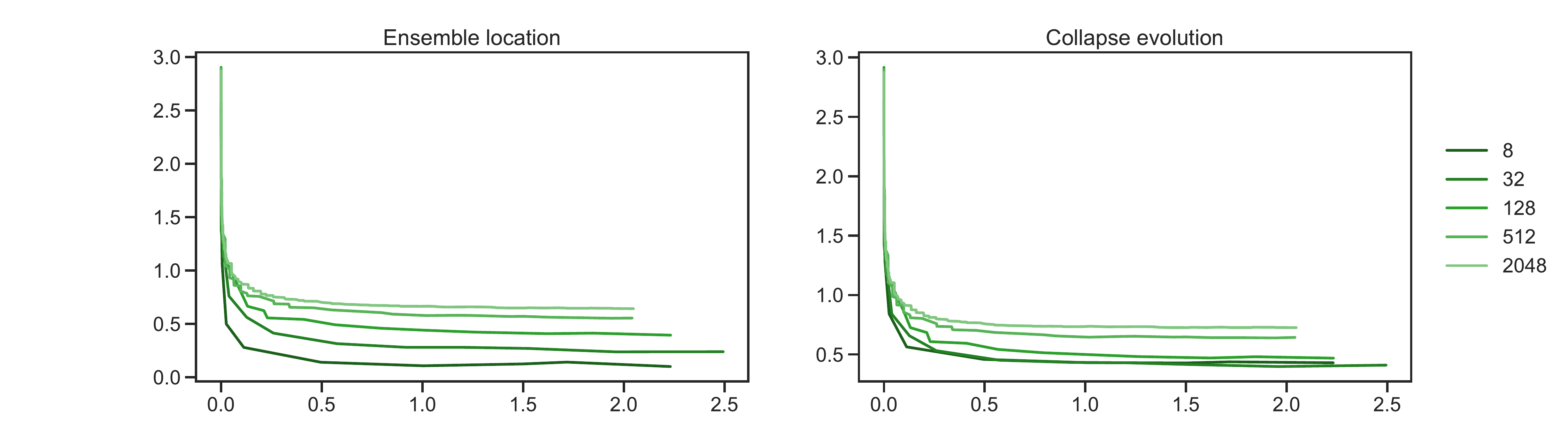}
  \label{sfig:darcy_negsob}}
  \subfigure[Evolution statistics of the EKS with respect to simulated time under norm $\|\cdot \|_{L^2}$.]{\includegraphics[width=\linewidth]{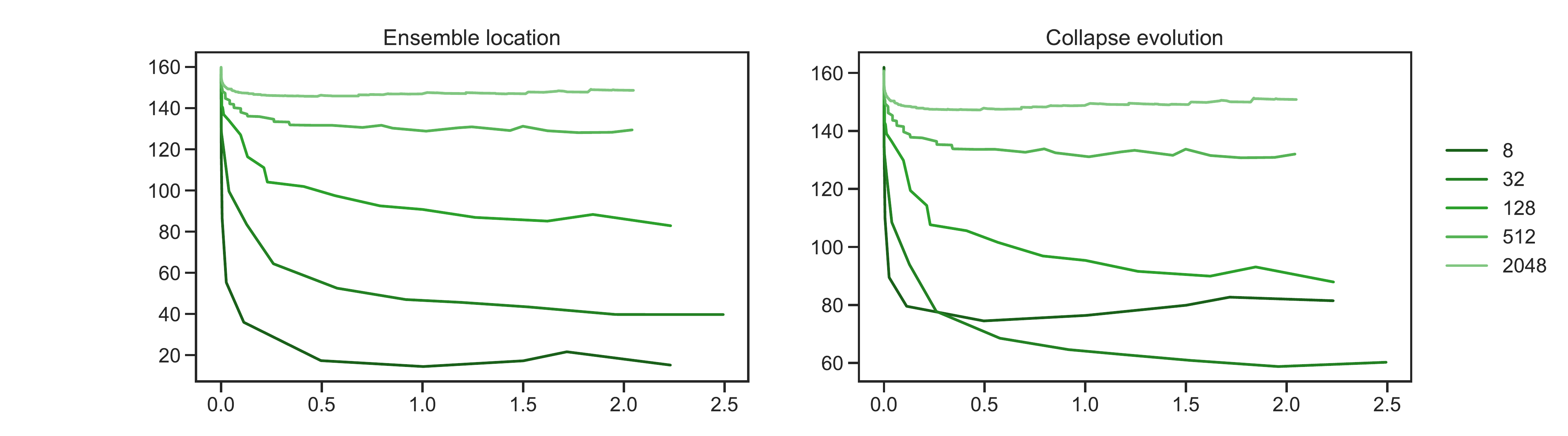}
  \label{sfig:darcy_L2}}
  \caption{Results for the Darcy flow inverse problem in high dimensions. The top
  panel shows scatter plots for different combinations of the higher modes
  in the KL expansion, \cref{eq:adefine}. The green dots correspond to the last
  iteration of the EKS at every ensemble size setting as labeled in the legend.
  Tracking the negative Sobolev norm of the ensembles with respect to its mean
  $\bar u(t)$ and underlying truth $u^\dagger$, shows good match to both the
  solution of the inverse problem and the stationary distribution of
  the Fokker-Planck equation, \cref{eq:Gibbs}. }
  \label{fig:darcy_contrast}
\end{figure}

Figure \ref{fig:darcy_identifiability} demonstrates
how different ensemble sizes are able to capture the marginal posterior
variances of each component in the unkown $u$. The top panel in Figure
\ref{fig:darcy_identifiability} tracks the posterior variance reduction
statistic for every component of $u' \in \R^d,$ which as mentioned before, now is viewed as a vector
of $d$ components rather than a function on subset $K_d$ of the
two-dimensional lattice. The posterior variance reduction
is a measure of the relative decrease of variance for a given quantity of
interest under the posterior with respect to the prior. It is defined as
\begin{align}
  \zeta_k = 1 - \frac{\mathbb{V}(u_k'|y)}{\mathbb{V}(u_k')},
\end{align}
where $\mathbb{V}(\cdot)$ denotes the variance of a random variable. The summary
statistic $\sum_k \zeta_k$ has been used in \citep{spiegelhalter2002bayesian} to
estimate the effective number of parameters in Bayesian models. When this
parameter is close to $1$ then the algorithm has reduced uncertainty
considerably, relative to the prior; when it is close to zero it has reduced it
very little, in comparison with the prior. By studying the figure for the MCMC
algorithm (orange) and comparing with EKS for increasing $J$ (green) we see that
for $J$ around $2000$ the match between EKS and MCMC is excellent. We also see
that for smaller-sized ensembles there is a regularizing effect which artificially
reduces the posterior variation for larger $k.$ On the other hand, the lower
panel in Figure~\ref{fig:darcy_identifiability} allows us to identify the
location of ensemble density by plotting the residuals $u_k' - (u_k^\dagger)'$, for
every component $k = 1, \ldots, d$; in particular we plot the algorithmic mean
of this quantity and $95\%$ confidence intervals. It can be seen that the
ensemble is well located as most of the components include the zero horizontal
line, meaning that marginally the distribution of every component includes the
correct value with high probability. Moreover, we can see two effects in this
figure. Firstly, the lower variability in the first components also shows that
there is enough information on the observed data to identify these components.
Secondly, it can be seen that for very low-sized ensembles the least important
components of $u$ incorporate higher error, when comparing the EKS samples in
green with the orange MCMC samples.
\begin{figure}[!t]
  \centering
  \includegraphics[width=\linewidth]{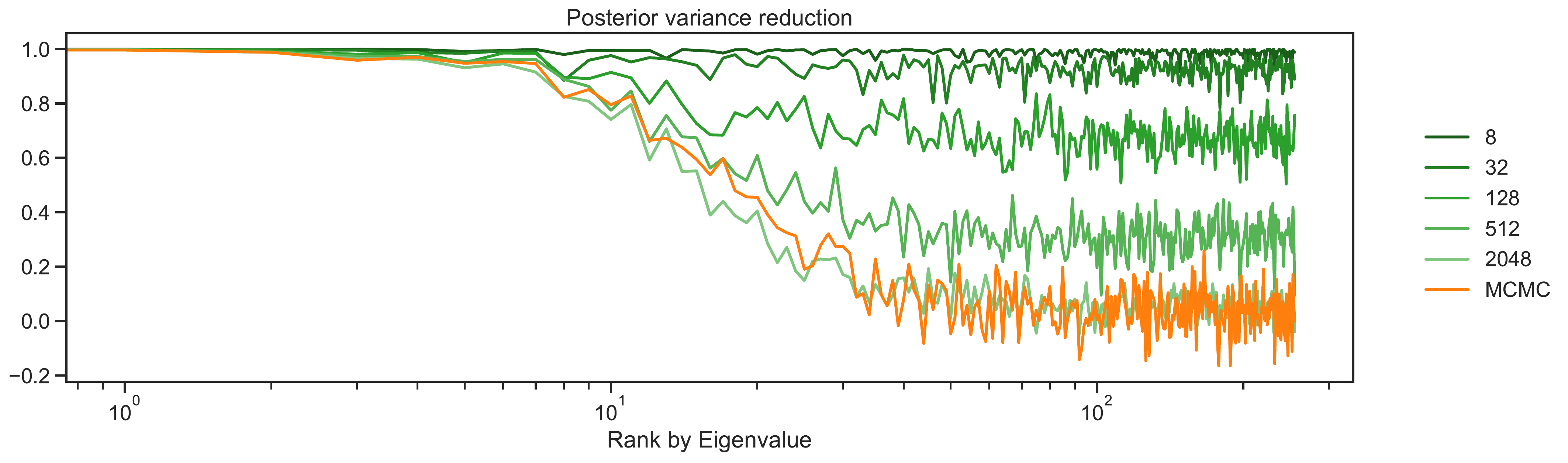}\\
  \medskip
  \includegraphics[width=\linewidth]{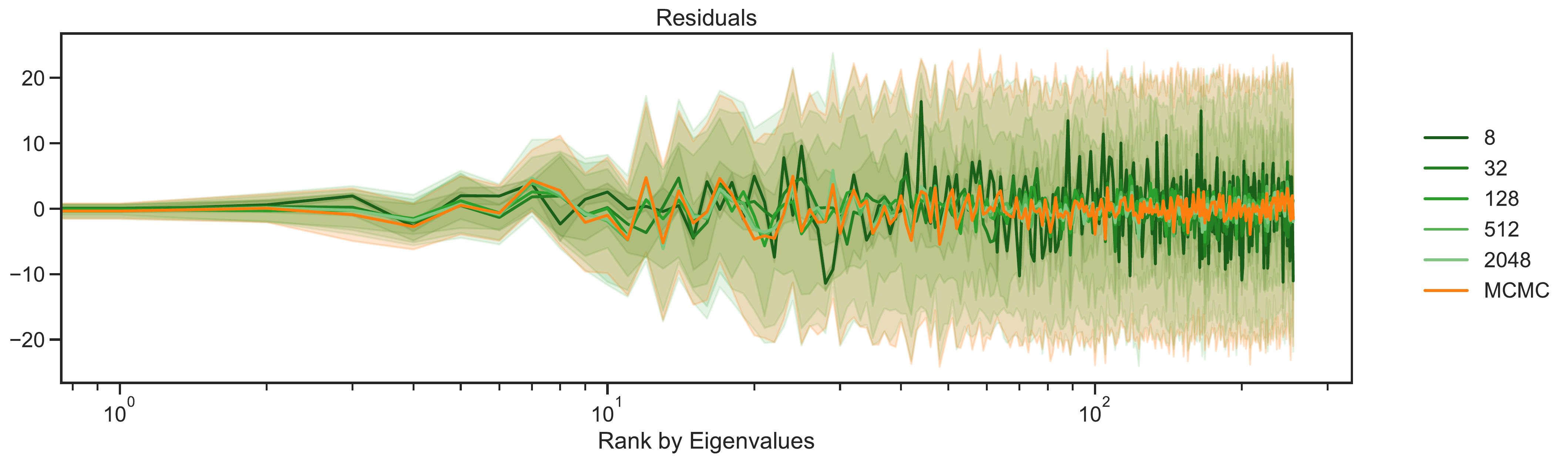}
  \caption{Results showing Darcy flow parameter identifiability. The top panel
  illustrates how bigger ensemble sizes are able to capture better the
  marginal posterior variability of each component. Whereas the lower panel,
  illustrates both variability and consistency of the approximate posterior
  samples from EKS. }
  \label{fig:darcy_identifiability}
\end{figure}

Overall, the mismatch between the results from EKS and the MCMC reference in
both numerical examples can be understood from the fact that the use of the
ensemble equations \eqref{eq:implement} introduces a linear approximation to the
curvature of the regularized misfit. This effect is demonstrated clearly in
Figure~\ref{fig:results}, which shows the samples from EKS against a background
of the level sets of the posterior. However, despite this mismatch, the key
point is that a relatively good set of approximate samples in green is computed
without use of the derivative of the forward model $\G$ in both numerical
examples; it thus holds promise as a method for large-scale nonlinear inverse
problems.

\section{Conclusions}
\label{sec:C}

In this paper we have demonstrated a methodoogy for the addition of  noise to
the basic EKI algorithm so that it generates approximate samples from the
Bayesian posterior distribution -- the ensemble Kalman sampler (EKS). Our
starting point is a set of interacting Langevin diffusions, preconditioned by
their mutual empirical covariance. To understand this system we introduce a new
mean-field Fokker-Planck equation which has the desired posterior distribution
as an invariant measure. We exhibit the new Kalman-Wasserstein metric with
respect to which the Fokker-Planck equation has gradient structure. We also show
how to compute approximate samples from this model by using a particle
approximation based on using ensemble differences in place of gradients, leading
to the EKS algorithm.

In the future we anticipate that methodology to correct for the error introduced
by use of ensemble differences will be a worthwhile development from the
algorithms proposed and we are actively pursuing this \citep{OLSpaper}.
Furthermore, recent interesting work of \cite{reichnew}
studies the invariant measures of the finite particle system
\eqref{eq:eki-basic-iter}, \eqref{eq:ec}. The authors identify a simple linear
correction term of order $J^{-1}$ in \eqref{eq:eki-basic-iter} which renders the
$J-$fold product of the posterior distribution invariant for finite ensemble
number; since one of the major motivations for the use of ensemble methods is
their robustness for small $J$, this correction is important.

We also recognize that other difference-based methods for approximating
gradients may emerge and that developing theory to quantify and control the
errors arising from such difference approximations will be of interest. We
believe that our proposed ensemble-based difference approximation is of
particular value because of the growing community of scientists and engineers
who work directly with ensemble based methods, because of their simplicity and
black-box nature. In the future, we will also study the properties of the
Kalman-Wasserstein metric including its duality, geodesics, and geometric
structure, a line of research that is of independent mathematical interest in
the context of generalized Wasserstein-type spaces. We will investigate the
analytical properties of the new metric within Gaussian families. We expect
these studies will bring insights to design new numerical algorithms for the
approximate solution of inverse problems.

\bigskip

{\footnotesize
\noindent
\textbf{Acknowledgements:}
The authors are grateful to Jos\'e A. Carrillo, Greg Pavliotis and
Sebastian Reich for helpful input which improved this paper. A.G.I. and A.M.S. are supported by the generosity of Eric and Wendy Schmidt by recommendation of the Schmidt Futures program, by Earthrise Alliance, by the Paul G. Allen Family Foundation, and by the National Science Foundation (NSF grant AGS‐1835860). A.M.S. is also supported by NSF under grant DMS 1818977. F.H. was partially supported by Caltech's von Karman postdoctoral instructorship. W.L. was supported by AFOSR MURI FA9550-18-1-0502.
}

\bibliography{main}

\begin{thebibliography}{98}

\bibitem[\protect\citeauthoryear{Ambrosio, Gigli and
  Savar\'e}{2005}]{AmbrosioGigliSavare2005_gradienta}
\begin{bmisc}[author]
\bauthor{\bsnm{Ambrosio},~\bfnm{Luigi}\binits{L.}},
  \bauthor{\bsnm{Gigli},~\bfnm{Nicola}\binits{N.}} \AND
  \bauthor{\bsnm{Savar\'e},~\bfnm{Giuseppe}\binits{G.}}
(\byear{2005}).
\btitle{Gradient {{Flows}}: In {{Metric Spaces}} and in the {{Space}} of
  {{Probability Measures}}}.
\end{bmisc}
\endbibitem

\bibitem[\protect\citeauthoryear{Arnold et~al.}{2001}]{AMTU}
\begin{barticle}[author]
\bauthor{\bsnm{Arnold},~\bfnm{Anton}\binits{A.}},
  \bauthor{\bsnm{Markowich},~\bfnm{Peter}\binits{P.}},
  \bauthor{\bsnm{Toscani},~\bfnm{Giuseppe}\binits{G.}} \AND
  \bauthor{\bsnm{Unterreiter},~\bfnm{Andreas}\binits{A.}}
(\byear{2001}).
\btitle{On convex {S}obolev inequalities and the rate of convergence to
  equilibrium for {F}okker-{P}lanck type equations}.
\bjournal{Comm. Partial Differential Equations}
\bvolume{26}
\bpages{43--100}.
\bdoi{10.1081/PDE-100002246}
\bmrnumber{1842428}
\end{barticle}
\endbibitem

\bibitem[\protect\citeauthoryear{Ay et~al.}{2017}]{IG2}
\begin{bbook}[author]
\bauthor{\bsnm{Ay},~\bfnm{Nihat}\binits{N.}},
  \bauthor{\bsnm{Jost},~\bfnm{J\"urgen}\binits{J.}},
  \bauthor{\bsnm{L\^e},~\bfnm{H\^ong~V\^an}\binits{H.~V.}} \AND
  \bauthor{\bsnm{Schwachh\"ofer},~\bfnm{Lorenz~Johannes}\binits{L.~J.}}
(\byear{2017}).
\btitle{Information Geometry}.
\bseries{Ergebnisse Der {{Mathematik}} Und Ihrer {{Grenzgebiete A}} @series of
  Modern Surveys in Mathematics\$l3. {{Folge}}, Volume 64}.
\bpublisher{{Springer}}, \baddress{Cham}.
\end{bbook}
\endbibitem

\bibitem[\protect\citeauthoryear{Bakry and \'{E}mery}{1985}]{BakryEmery}
\begin{bincollection}[author]
\bauthor{\bsnm{Bakry},~\bfnm{D.}\binits{D.}} \AND
  \bauthor{\bsnm{\'{E}mery},~\bfnm{Michel}\binits{M.}}
(\byear{1985}).
\btitle{Diffusions hypercontractives}.
In \bbooktitle{S\'{e}minaire de probabilit\'{e}s, {XIX}, 1983/84}.
\bseries{Lecture Notes in Math.}
\bvolume{1123}
\bpages{177--206}.
\bpublisher{Springer, Berlin}.
\bdoi{10.1007/BFb0075847}
\bmrnumber{889476}
\end{bincollection}
\endbibitem

\bibitem[\protect\citeauthoryear{Bedard}{2008}]{bedard2008optimal}
\begin{barticle}[author]
\bauthor{\bsnm{Bedard},~\bfnm{Mylene}\binits{M.}}
(\byear{2008}).
\btitle{Optimal acceptance rates for {Metropolis} algorithms: Moving beyond
  0.234}.
\bjournal{Stochastic Processes and their Applications}
\bvolume{118}
\bpages{2198--2222}.
\end{barticle}
\endbibitem

\bibitem[\protect\citeauthoryear{B{\'e}dard et~al.}{2007}]{bedard2007weak}
\begin{barticle}[author]
\bauthor{\bsnm{B{\'e}dard},~\bfnm{Mylene}\binits{M.}} \betal{et~al.}
(\byear{2007}).
\btitle{Weak convergence of {Metropolis} algorithms for non-iid target
  distributions}.
\bjournal{The Annals of Applied Probability}
\bvolume{17}
\bpages{1222--1244}.
\end{barticle}
\endbibitem

\bibitem[\protect\citeauthoryear{B{\'e}dard and
  Rosenthal}{2008}]{bedard2008optimalb}
\begin{barticle}[author]
\bauthor{\bsnm{B{\'e}dard},~\bfnm{Mylene}\binits{M.}} \AND
  \bauthor{\bsnm{Rosenthal},~\bfnm{Jeffrey~S}\binits{J.~S.}}
(\byear{2008}).
\btitle{Optimal scaling of {Metropolis} algorithms: Heading toward general
  target distributions}.
\bjournal{Canadian Journal of Statistics}
\bvolume{36}
\bpages{483--503}.
\end{barticle}
\endbibitem

\bibitem[\protect\citeauthoryear{Bergemann and
  Reich}{2010a}]{bergemann2010localization}
\begin{barticle}[author]
\bauthor{\bsnm{Bergemann},~\bfnm{Kay}\binits{K.}} \AND
  \bauthor{\bsnm{Reich},~\bfnm{Sebastian}\binits{S.}}
(\byear{2010}a).
\btitle{A localization technique for ensemble {Kalman} filters}.
\bjournal{Quarterly Journal of the Royal Meteorological Society}
\bvolume{136}
\bpages{701--707}.
\end{barticle}
\endbibitem

\bibitem[\protect\citeauthoryear{Bergemann and
  Reich}{2010b}]{bergemann2010mollified}
\begin{barticle}[author]
\bauthor{\bsnm{Bergemann},~\bfnm{Kay}\binits{K.}} \AND
  \bauthor{\bsnm{Reich},~\bfnm{Sebastian}\binits{S.}}
(\byear{2010}b).
\btitle{A mollified ensemble {Kalman} filter}.
\bjournal{Quarterly Journal of the Royal Meteorological Society}
\bvolume{136}
\bpages{1636--1643}.
\end{barticle}
\endbibitem

\bibitem[\protect\citeauthoryear{Bergemann and
  Reich}{2012}]{bergemann2012ensemble}
\begin{barticle}[author]
\bauthor{\bsnm{Bergemann},~\bfnm{Kay}\binits{K.}} \AND
  \bauthor{\bsnm{Reich},~\bfnm{Sebastian}\binits{S.}}
(\byear{2012}).
\btitle{An ensemble {Kalman-Bucy} filter for continuous data assimilation}.
\bjournal{Meteorologische Zeitschrift}
\bvolume{21}
\bpages{213--219}.
\end{barticle}
\endbibitem

\bibitem[\protect\citeauthoryear{Carrassi et~al.}{2018}]{carrassi2018data}
\begin{barticle}[author]
\bauthor{\bsnm{Carrassi},~\bfnm{Alberto}\binits{A.}},
  \bauthor{\bsnm{Bocquet},~\bfnm{Marc}\binits{M.}},
  \bauthor{\bsnm{Bertino},~\bfnm{Laurent}\binits{L.}} \AND
  \bauthor{\bsnm{Evensen},~\bfnm{Geir}\binits{G.}}
(\byear{2018}).
\btitle{Data assimilation in the geosciences: An overview of methods, issues,
  and perspectives}.
\bjournal{Wiley Interdisciplinary Reviews: Climate Change}
\bvolume{9}
\bpages{e535}.
\end{barticle}
\endbibitem

\bibitem[\protect\citeauthoryear{Carrillo and
  Toscani}{1998}]{CarrilloToscani1998}
\begin{barticle}[author]
\bauthor{\bsnm{Carrillo},~\bfnm{J.~A.}\binits{J.~A.}} \AND
  \bauthor{\bsnm{Toscani},~\bfnm{G.}\binits{G.}}
(\byear{1998}).
\btitle{Exponential convergence toward equilibrium for homogeneous
  {F}okker-{P}lanck-type equations}.
\bjournal{Math. Methods Appl. Sci.}
\bvolume{21}
\bpages{1269--1286}.
\bdoi{10.1002/(SICI)1099-1476(19980910)21:13<1269::AID-MMA995>3.3.CO;2-F}
\bmrnumber{1639292}
\end{barticle}
\endbibitem

\bibitem[\protect\citeauthoryear{Carrillo et~al.}{2010}]{CFTV10}
\begin{bincollection}[author]
\bauthor{\bsnm{Carrillo},~\bfnm{Jos\'{e}~A.}\binits{J.~A.}},
  \bauthor{\bsnm{Fornasier},~\bfnm{Massimo}\binits{M.}},
  \bauthor{\bsnm{Toscani},~\bfnm{Giuseppe}\binits{G.}} \AND
  \bauthor{\bsnm{Vecil},~\bfnm{Francesco}\binits{F.}}
(\byear{2010}).
\btitle{Particle, kinetic, and hydrodynamic models of swarming}.
In \bbooktitle{Mathematical modeling of collective behavior in socio-economic
  and life sciences}.
\bseries{Model. Simul. Sci. Eng. Technol.}
\bpages{297--336}.
\bpublisher{Birkh\"{a}user Boston, Inc., Boston, MA}.
\bdoi{10.1007/978-0-8176-4946-3_12}
\bmrnumber{2744704}
\end{bincollection}
\endbibitem

\bibitem[\protect\citeauthoryear{Carrillo
  et~al.}{2018}]{carrillo2018analytical}
\begin{barticle}[author]
\bauthor{\bsnm{Carrillo},~\bfnm{Jos{\'e}~A}\binits{J.~A.}},
  \bauthor{\bsnm{Choi},~\bfnm{Young-Pil}\binits{Y.-P.}},
  \bauthor{\bsnm{Totzeck},~\bfnm{Claudia}\binits{C.}} \AND
  \bauthor{\bsnm{Tse},~\bfnm{Oliver}\binits{O.}}
(\byear{2018}).
\btitle{An analytical framework for consensus-based global optimization
  method}.
\bjournal{Mathematical Models and Methods in Applied Sciences}
\bvolume{28}
\bpages{1037--1066}.
\end{barticle}
\endbibitem

\bibitem[\protect\citeauthoryear{Chada, Stuart and
  Tong}{2019}]{chada2019tikhonov}
\begin{barticle}[author]
\bauthor{\bsnm{Chada},~\bfnm{Neil~K}\binits{N.~K.}},
  \bauthor{\bsnm{Stuart},~\bfnm{Andrew~M}\binits{A.~M.}} \AND
  \bauthor{\bsnm{Tong},~\bfnm{Xin~T}\binits{X.~T.}}
(\byear{2019}).
\btitle{Tikhonov Regularization Within Ensemble {Kalman} Inversion}.
\bjournal{arXiv preprint arXiv:1901.10382}.
\end{barticle}
\endbibitem

\bibitem[\protect\citeauthoryear{Chen and Oliver}{2012}]{chen2012ensemble}
\begin{barticle}[author]
\bauthor{\bsnm{Chen},~\bfnm{Yan}\binits{Y.}} \AND
  \bauthor{\bsnm{Oliver},~\bfnm{Dean~S}\binits{D.~S.}}
(\byear{2012}).
\btitle{Ensemble randomized maximum likelihood method as an iterative ensemble
  smoother}.
\bjournal{Mathematical Geosciences}
\bvolume{44}
\bpages{1--26}.
\end{barticle}
\endbibitem

\bibitem[\protect\citeauthoryear{Cleary et~al.}{2019}]{OLSpaper}
\begin{barticle}[author]
\bauthor{\bsnm{Cleary},~\bfnm{Emmet}\binits{E.}},
  \bauthor{\bsnm{Garbuno-Inigo},~\bfnm{Alfredo}\binits{A.}},
  \bauthor{\bsnm{Lan},~\bfnm{Shiwei}\binits{S.}},
  \bauthor{\bsnm{Schneider},~\bfnm{Tapio}\binits{T.}} \AND
  \bauthor{\bsnm{Stuart},~\bfnm{Andrew~M.}\binits{A.~M.}}
(\byear{2019}).
\btitle{Calibrate, Emulate, Sample}.
\bjournal{arXiv preprint arXiv:1912.}
\end{barticle}
\endbibitem

\bibitem[\protect\citeauthoryear{Cotter et~al.}{2013}]{cotter2013mcmc}
\begin{barticle}[author]
\bauthor{\bsnm{Cotter},~\bfnm{Simon~L}\binits{S.~L.}},
  \bauthor{\bsnm{Roberts},~\bfnm{Gareth~O}\binits{G.~O.}},
  \bauthor{\bsnm{Stuart},~\bfnm{Andrew~M}\binits{A.~M.}},
  \bauthor{\bsnm{White},~\bfnm{David}\binits{D.}} \betal{et~al.}
(\byear{2013}).
\btitle{{MCMC} methods for functions: modifying old algorithms to make them
  faster}.
\bjournal{Statistical Science}
\bvolume{28}
\bpages{424--446}.
\end{barticle}
\endbibitem

\bibitem[\protect\citeauthoryear{Crisan and
  Xiong}{2010}]{crisan2010approximate}
\begin{barticle}[author]
\bauthor{\bsnm{Crisan},~\bfnm{Dan}\binits{D.}} \AND
  \bauthor{\bsnm{Xiong},~\bfnm{Jie}\binits{J.}}
(\byear{2010}).
\btitle{Approximate {M}c{K}ean--{V}lasov representations for a class of
  {SPDEs}}.
\bjournal{Stochastics An International Journal of Probability and Stochastics
  Processes}
\bvolume{82}
\bpages{53--68}.
\end{barticle}
\endbibitem

\bibitem[\protect\citeauthoryear{Daum and Huang}{2011}]{daum2011particle}
\begin{binproceedings}[author]
\bauthor{\bsnm{Daum},~\bfnm{Fred}\binits{F.}} \AND
  \bauthor{\bsnm{Huang},~\bfnm{Jim}\binits{J.}}
(\byear{2011}).
\btitle{Particle flow for nonlinear filters}.
In \bbooktitle{2011 IEEE International Conference on Acoustics, Speech and
  Signal Processing (ICASSP)}
\bpages{5920--5923}.
\bpublisher{IEEE}.
\end{binproceedings}
\endbibitem

\bibitem[\protect\citeauthoryear{de~Wiljes, Reich and
  Stannat}{2018}]{de2018long}
\begin{barticle}[author]
\bauthor{\bparticle{de} \bsnm{Wiljes},~\bfnm{Jana}\binits{J.}},
  \bauthor{\bsnm{Reich},~\bfnm{Sebastian}\binits{S.}} \AND
  \bauthor{\bsnm{Stannat},~\bfnm{Wilhelm}\binits{W.}}
(\byear{2018}).
\btitle{Long-Time Stability and Accuracy of the Ensemble {Kalman--Bucy} Filter
  for Fully Observed Processes and Small Measurement Noise}.
\bjournal{SIAM Journal on Applied Dynamical Systems}
\bvolume{17}
\bpages{1152--1181}.
\end{barticle}
\endbibitem

\bibitem[\protect\citeauthoryear{Del~Moral, Kurtzmann and
  Tugaut}{2017}]{del2017stability}
\begin{barticle}[author]
\bauthor{\bsnm{Del~Moral},~\bfnm{Pierre}\binits{P.}},
  \bauthor{\bsnm{Kurtzmann},~\bfnm{Aline}\binits{A.}} \AND
  \bauthor{\bsnm{Tugaut},~\bfnm{Julian}\binits{J.}}
(\byear{2017}).
\btitle{On the Stability and the Uniform Propagation of Chaos of a Class of
  Extended Ensemble {Kalman--Bucy} Filters}.
\bjournal{SIAM Journal on Control and Optimization}
\bvolume{55}
\bpages{119--155}.
\end{barticle}
\endbibitem

\bibitem[\protect\citeauthoryear{Del~Moral et~al.}{2018}]{del2018stability}
\begin{barticle}[author]
\bauthor{\bsnm{Del~Moral},~\bfnm{Pierre}\binits{P.}},
  \bauthor{\bsnm{Tugaut},~\bfnm{Julian}\binits{J.}} \betal{et~al.}
(\byear{2018}).
\btitle{On the stability and the uniform propagation of chaos properties of
  ensemble {Kalman--Bucy} filters}.
\bjournal{The Annals of Applied Probability}
\bvolume{28}
\bpages{790--850}.
\end{barticle}
\endbibitem

\bibitem[\protect\citeauthoryear{Detommaso et~al.}{2018}]{detommaso2018stein}
\begin{binproceedings}[author]
\bauthor{\bsnm{Detommaso},~\bfnm{Gianluca}\binits{G.}},
  \bauthor{\bsnm{Cui},~\bfnm{Tiangang}\binits{T.}},
  \bauthor{\bsnm{Marzouk},~\bfnm{Youssef}\binits{Y.}},
  \bauthor{\bsnm{Spantini},~\bfnm{Alessio}\binits{A.}} \AND
  \bauthor{\bsnm{Scheichl},~\bfnm{Robert}\binits{R.}}
(\byear{2018}).
\btitle{A {Stein} variational {Newton} method}.
In \bbooktitle{Advances in Neural Information Processing Systems}
\bpages{9187--9197}.
\end{binproceedings}
\endbibitem

\bibitem[\protect\citeauthoryear{Ding and Li}{2019}]{ding2019mean}
\begin{barticle}[author]
\bauthor{\bsnm{Ding},~\bfnm{Zhiyan}\binits{Z.}} \AND
  \bauthor{\bsnm{Li},~\bfnm{Qin}\binits{Q.}}
(\byear{2019}).
\btitle{Mean-field limit and numerical analysis for Ensemble Kalman Inversion:
  linear setting}.
\bjournal{arXiv preprint arXiv:1908.05575}.
\end{barticle}
\endbibitem

\bibitem[\protect\citeauthoryear{Duncan, Lelievre and
  Pavliotis}{2016}]{duncan2016variance}
\begin{barticle}[author]
\bauthor{\bsnm{Duncan},~\bfnm{Andrew~B}\binits{A.~B.}},
  \bauthor{\bsnm{Lelievre},~\bfnm{Tony}\binits{T.}} \AND
  \bauthor{\bsnm{Pavliotis},~\bfnm{GA}\binits{G.}}
(\byear{2016}).
\btitle{Variance reduction using nonreversible Langevin samplers}.
\bjournal{Journal of Statistical Physics}
\bvolume{163}
\bpages{457--491}.
\end{barticle}
\endbibitem

\bibitem[\protect\citeauthoryear{Duncan and Szpruch}{2019}]{Duncan}
\begin{barticle}[author]
\bauthor{\bsnm{Duncan},~\bfnm{Andrew}\binits{A.}} \AND
  \bauthor{\bsnm{Szpruch},~\bfnm{Lukasz}\binits{L.}}
(\byear{2019}).
\btitle{Private Communication}.
\end{barticle}
\endbibitem

\bibitem[\protect\citeauthoryear{El~Moselhy and Marzouk}{2012}]{el2012bayesian}
\begin{barticle}[author]
\bauthor{\bsnm{El~Moselhy},~\bfnm{Tarek~A}\binits{T.~A.}} \AND
  \bauthor{\bsnm{Marzouk},~\bfnm{Youssef~M}\binits{Y.~M.}}
(\byear{2012}).
\btitle{Bayesian inference with optimal maps}.
\bjournal{Journal of Computational Physics}
\bvolume{231}
\bpages{7815--7850}.
\end{barticle}
\endbibitem

\bibitem[\protect\citeauthoryear{Emerick and
  Reynolds}{2013}]{emerick2013investigation}
\begin{barticle}[author]
\bauthor{\bsnm{Emerick},~\bfnm{Alexandre~A}\binits{A.~A.}} \AND
  \bauthor{\bsnm{Reynolds},~\bfnm{Albert~C}\binits{A.~C.}}
(\byear{2013}).
\btitle{Investigation of the sampling performance of ensemble-based methods
  with a simple reservoir model}.
\bjournal{Computational Geosciences}
\bvolume{17}
\bpages{325--350}.
\end{barticle}
\endbibitem

\bibitem[\protect\citeauthoryear{Engl, Hanke and
  Neubauer}{1996}]{engl1996regularization}
\begin{bbook}[author]
\bauthor{\bsnm{Engl},~\bfnm{Heinz~Werner}\binits{H.~W.}},
  \bauthor{\bsnm{Hanke},~\bfnm{Martin}\binits{M.}} \AND
  \bauthor{\bsnm{Neubauer},~\bfnm{Andreas}\binits{A.}}
(\byear{1996}).
\btitle{Regularization of inverse problems}
\bvolume{375}.
\bpublisher{Springer Science \& Business Media}.
\end{bbook}
\endbibitem

\bibitem[\protect\citeauthoryear{Ernst, Sprungk and
  Starkloff}{2015}]{ernst2015analysis}
\begin{barticle}[author]
\bauthor{\bsnm{Ernst},~\bfnm{Oliver~G}\binits{O.~G.}},
  \bauthor{\bsnm{Sprungk},~\bfnm{Bj\"orn}\binits{B.}} \AND
  \bauthor{\bsnm{Starkloff},~\bfnm{Hans-J\"org}\binits{H.-J.}}
(\byear{2015}).
\btitle{Analysis of the ensemble and polynomial chaos {Kalman} filters in
  {Bayesian} inverse problems}.
\bjournal{SIAM/ASA Journal on Uncertainty Quantification}
\bvolume{3}
\bpages{823--851}.
\end{barticle}
\endbibitem

\bibitem[\protect\citeauthoryear{Evensen}{2009}]{evensen2009data}
\begin{bbook}[author]
\bauthor{\bsnm{Evensen},~\bfnm{Geir}\binits{G.}}
(\byear{2009}).
\btitle{Data assimilation: the ensemble {Kalman} filter}.
\bpublisher{Springer Science \& Business Media}.
\end{bbook}
\endbibitem

\bibitem[\protect\citeauthoryear{Girolami and
  Calderhead}{2011}]{girolami2011riemann}
\begin{barticle}[author]
\bauthor{\bsnm{Girolami},~\bfnm{Mark}\binits{M.}} \AND
  \bauthor{\bsnm{Calderhead},~\bfnm{Ben}\binits{B.}}
(\byear{2011}).
\btitle{Riemann manifold Langevin and Hamiltonian Monte Carlo methods}.
\bjournal{Journal of the Royal Statistical Society: Series B (Statistical
  Methodology)}
\bvolume{73}
\bpages{123--214}.
\end{barticle}
\endbibitem

\bibitem[\protect\citeauthoryear{Gonzalez}{1996}]{gonzalez1996time}
\begin{barticle}[author]
\bauthor{\bsnm{Gonzalez},~\bfnm{Oscar}\binits{O.}}
(\byear{1996}).
\btitle{Time integration and discrete {H}amiltonian systems}.
\bjournal{Journal of Nonlinear Science}
\bvolume{6}
\bpages{449}.
\end{barticle}
\endbibitem

\bibitem[\protect\citeauthoryear{Goodman and Weare}{2010}]{goodman2010ensemble}
\begin{barticle}[author]
\bauthor{\bsnm{Goodman},~\bfnm{Jonathan}\binits{J.}} \AND
  \bauthor{\bsnm{Weare},~\bfnm{Jonathan}\binits{J.}}
(\byear{2010}).
\btitle{Ensemble samplers with affine invariance}.
\bjournal{Communications in applied mathematics and computational science}
\bvolume{5}
\bpages{65--80}.
\end{barticle}
\endbibitem

\bibitem[\protect\citeauthoryear{Ha and Tadmor}{2008}]{HaTadmor08}
\begin{barticle}[author]
\bauthor{\bsnm{Ha},~\bfnm{Seung-Yeal}\binits{S.-Y.}} \AND
  \bauthor{\bsnm{Tadmor},~\bfnm{Eitan}\binits{E.}}
(\byear{2008}).
\btitle{From particle to kinetic and hydrodynamic descriptions of flocking}.
\bjournal{Kinet. Relat. Models}
\bvolume{1}
\bpages{415--435}.
\bdoi{10.3934/krm.2008.1.415}
\bmrnumber{2425606}
\end{barticle}
\endbibitem

\bibitem[\protect\citeauthoryear{Hairer and Lubich}{2013}]{hairer2013energy}
\begin{barticle}[author]
\bauthor{\bsnm{Hairer},~\bfnm{Ernst}\binits{E.}} \AND
  \bauthor{\bsnm{Lubich},~\bfnm{Christian}\binits{C.}}
(\byear{2013}).
\btitle{Energy-diminishing integration of gradient systems}.
\bjournal{IMA Journal of Numerical Analysis}
\bvolume{34}
\bpages{452--461}.
\end{barticle}
\endbibitem

\bibitem[\protect\citeauthoryear{Halder and Georgiou}{2017}]{HG}
\begin{barticle}[author]
\bauthor{\bsnm{Halder},~\bfnm{Abhishek}\binits{A.}} \AND
  \bauthor{\bsnm{Georgiou},~\bfnm{Tryphon~T.}\binits{T.~T.}}
(\byear{2017}).
\btitle{Gradient {{Flows}} in {{Filtering}} and {{Fisher}}-{{Rao Geometry}}}.
\bjournal{arXiv:1710.00064 [cs, math]}.
\end{barticle}
\endbibitem

\bibitem[\protect\citeauthoryear{Herty and Visconti}{2018}]{herty2018kinetic}
\begin{barticle}[author]
\bauthor{\bsnm{Herty},~\bfnm{Michael}\binits{M.}} \AND
  \bauthor{\bsnm{Visconti},~\bfnm{Giuseppe}\binits{G.}}
(\byear{2018}).
\btitle{Kinetic Methods for Inverse Problems}.
\bjournal{arXiv preprint arXiv:1811.09387}.
\end{barticle}
\endbibitem

\bibitem[\protect\citeauthoryear{Humphries and
  Stuart}{1994}]{humphries1994runge}
\begin{barticle}[author]
\bauthor{\bsnm{Humphries},~\bfnm{AR}\binits{A.}} \AND
  \bauthor{\bsnm{Stuart},~\bfnm{AM}\binits{A.}}
(\byear{1994}).
\btitle{Runge--{K}utta methods for dissipative and gradient dynamical systems}.
\bjournal{SIAM journal on numerical analysis}
\bvolume{31}
\bpages{1452--1485}.
\end{barticle}
\endbibitem

\bibitem[\protect\citeauthoryear{Iglesias}{2015}]{iglesias2015iterative}
\begin{barticle}[author]
\bauthor{\bsnm{Iglesias},~\bfnm{Marco~A}\binits{M.~A.}}
(\byear{2015}).
\btitle{Iterative regularization for ensemble data assimilation in reservoir
  models}.
\bjournal{Computational Geosciences}
\bvolume{19}
\bpages{177--212}.
\end{barticle}
\endbibitem

\bibitem[\protect\citeauthoryear{Iglesias}{2016}]{iglesias2016regularizing}
\begin{barticle}[author]
\bauthor{\bsnm{Iglesias},~\bfnm{Marco~A}\binits{M.~A.}}
(\byear{2016}).
\btitle{A regularizing iterative ensemble {Kalman} method for {PDE}-constrained
  inverse problems}.
\bjournal{Inverse Problems}
\bvolume{32}
\bpages{025002}.
\end{barticle}
\endbibitem

\bibitem[\protect\citeauthoryear{Iglesias, Law and
  Stuart}{2013}]{iglesias2013ensemble}
\begin{barticle}[author]
\bauthor{\bsnm{Iglesias},~\bfnm{Marco~A}\binits{M.~A.}},
  \bauthor{\bsnm{Law},~\bfnm{Kody~JH}\binits{K.~J.}} \AND
  \bauthor{\bsnm{Stuart},~\bfnm{Andrew~M}\binits{A.~M.}}
(\byear{2013}).
\btitle{Ensemble {Kalman} methods for inverse problems}.
\bjournal{Inverse Problems}
\bvolume{29}
\bpages{045001}.
\end{barticle}
\endbibitem

\bibitem[\protect\citeauthoryear{Jabin and Wang}{2017}]{Jabin2017}
\begin{binbook}[author]
\bauthor{\bsnm{Jabin},~\bfnm{Pierre-Emmanuel}\binits{P.-E.}} \AND
  \bauthor{\bsnm{Wang},~\bfnm{Zhenfu}\binits{Z.}}
(\byear{2017}).
\btitle{Mean Field Limit for Stochastic Particle Systems}
In \bbooktitle{Active Particles, Volume 1 : Advances in Theory, Models, and
  Applications}
\bpages{379--402}.
\bpublisher{Springer International Publishing}, \baddress{Cham}.
\bdoi{10.1007/978-3-319-49996-3_10}
\end{binbook}
\endbibitem

\bibitem[\protect\citeauthoryear{Jordan, Kinderlehrer and Otto}{1998}]{JKO}
\begin{barticle}[author]
\bauthor{\bsnm{Jordan},~\bfnm{Richard}\binits{R.}},
  \bauthor{\bsnm{Kinderlehrer},~\bfnm{David}\binits{D.}} \AND
  \bauthor{\bsnm{Otto},~\bfnm{Felix}\binits{F.}}
(\byear{1998}).
\btitle{The {{Variational Formulation}} of the {{Fokker}}--{{Planck
  Equation}}}.
\bjournal{SIAM Journal on Mathematical Analysis}
\bvolume{29}
\bpages{1-17}.
\end{barticle}
\endbibitem

\bibitem[\protect\citeauthoryear{Kaipio and
  Somersalo}{2006}]{kaipio2006statistical}
\begin{bbook}[author]
\bauthor{\bsnm{Kaipio},~\bfnm{Jari}\binits{J.}} \AND
  \bauthor{\bsnm{Somersalo},~\bfnm{Erkki}\binits{E.}}
(\byear{2006}).
\btitle{Statistical and computational inverse problems}
\bvolume{160}.
\bpublisher{Springer Science \& Business Media}.
\end{bbook}
\endbibitem

\bibitem[\protect\citeauthoryear{Kalman}{1960}]{kalman1960new}
\begin{barticle}[author]
\bauthor{\bsnm{Kalman},~\bfnm{Rudolph~Emil}\binits{R.~E.}}
(\byear{1960}).
\btitle{A new approach to linear filtering and prediction problems}.
\bjournal{Journal of basic Engineering}
\bvolume{82}
\bpages{35--45}.
\end{barticle}
\endbibitem

\bibitem[\protect\citeauthoryear{Kalman and Bucy}{1961}]{KB}
\begin{barticle}[author]
\bauthor{\bsnm{Kalman},~\bfnm{R.~E.}\binits{R.~E.}} \AND
  \bauthor{\bsnm{Bucy},~\bfnm{R.~S.}\binits{R.~S.}}
(\byear{1961}).
\btitle{New {{Results}} in {{Linear Filtering}} and {{Prediction Theory}}}.
\bjournal{Journal of Basic Engineering}
\bvolume{83}
\bpages{95}.
\end{barticle}
\endbibitem

\bibitem[\protect\citeauthoryear{Kelly, Law and Stuart}{2014}]{kelly2014well}
\begin{barticle}[author]
\bauthor{\bsnm{Kelly},~\bfnm{David~TB}\binits{D.~T.}},
  \bauthor{\bsnm{Law},~\bfnm{KJH}\binits{K.}} \AND
  \bauthor{\bsnm{Stuart},~\bfnm{Andrew~M}\binits{A.~M.}}
(\byear{2014}).
\btitle{Well-posedness and accuracy of the ensemble {Kalman} filter in discrete
  and continuous time}.
\bjournal{Nonlinearity}
\bvolume{27}
\bpages{2579}.
\end{barticle}
\endbibitem

\bibitem[\protect\citeauthoryear{Kovachki and
  Stuart}{2018}]{KovachkiStuart2018_ensemble}
\begin{barticle}[author]
\bauthor{\bsnm{Kovachki},~\bfnm{Nikola~B.}\binits{N.~B.}} \AND
  \bauthor{\bsnm{Stuart},~\bfnm{Andrew~M.}\binits{A.~M.}}
(\byear{2018}).
\btitle{Ensemble {{Kalman Inversion}}: {{A Derivative}}-{{Free Technique For
  Machine Learning Tasks}}}.
\bjournal{arXiv:1808.03620 [cs, math, stat]}.
\end{barticle}
\endbibitem

\bibitem[\protect\citeauthoryear{Lafferty}{1988}]{Lafferty}
\begin{barticle}[author]
\bauthor{\bsnm{Lafferty},~\bfnm{John~D.}\binits{J.~D.}}
(\byear{1988}).
\btitle{The Density Manifold and Configuration Space Quantization}.
\bjournal{Transactions of the American Mathematical Society}
\bvolume{305}
\bpages{699-741}.
\end{barticle}
\endbibitem

\bibitem[\protect\citeauthoryear{Laugesen et~al.}{2015}]{LMMR}
\begin{barticle}[author]
\bauthor{\bsnm{Laugesen},~\bfnm{Richard~S.}\binits{R.~S.}},
  \bauthor{\bsnm{Mehta},~\bfnm{Prashant~G.}\binits{P.~G.}},
  \bauthor{\bsnm{Meyn},~\bfnm{Sean~P.}\binits{S.~P.}} \AND
  \bauthor{\bsnm{Raginsky},~\bfnm{Maxim}\binits{M.}}
(\byear{2015}).
\btitle{Poisson's {{Equation}} in {{Nonlinear Filtering}}}.
\bjournal{SIAM Journal on Control and Optimization}
\bvolume{53}
\bpages{501-525}.
\end{barticle}
\endbibitem

\bibitem[\protect\citeauthoryear{Law, Stuart and Zygalakis}{2015}]{law2015data}
\begin{bbook}[author]
\bauthor{\bsnm{Law},~\bfnm{Kody}\binits{K.}},
  \bauthor{\bsnm{Stuart},~\bfnm{Andrew}\binits{A.}} \AND
  \bauthor{\bsnm{Zygalakis},~\bfnm{Kostas}\binits{K.}}
(\byear{2015}).
\btitle{Data Assimilation: A Mathematical Introduction}.
\end{bbook}
\endbibitem

\bibitem[\protect\citeauthoryear{Leimkuhler and
  Matthews}{2016}]{leimkuhler2016molecular}
\begin{bbook}[author]
\bauthor{\bsnm{Leimkuhler},~\bfnm{Ben}\binits{B.}} \AND
  \bauthor{\bsnm{Matthews},~\bfnm{Charles}\binits{C.}}
(\byear{2016}).
\btitle{Molecular Dynamics}.
\bpublisher{Springer}.
\end{bbook}
\endbibitem

\bibitem[\protect\citeauthoryear{Leimkuhler, Matthews and
  Weare}{2018}]{leimkuhler2018ensemble}
\begin{barticle}[author]
\bauthor{\bsnm{Leimkuhler},~\bfnm{Benedict}\binits{B.}},
  \bauthor{\bsnm{Matthews},~\bfnm{Charles}\binits{C.}} \AND
  \bauthor{\bsnm{Weare},~\bfnm{Jonathan}\binits{J.}}
(\byear{2018}).
\btitle{Ensemble preconditioning for Markov chain Monte Carlo simulation}.
\bjournal{Statistics and Computing}
\bvolume{28}
\bpages{277--290}.
\end{barticle}
\endbibitem

\bibitem[\protect\citeauthoryear{Leimkuhler, Noorizadeh and
  Theil}{2009}]{leimkuhler2009gentle}
\begin{barticle}[author]
\bauthor{\bsnm{Leimkuhler},~\bfnm{Ben}\binits{B.}},
  \bauthor{\bsnm{Noorizadeh},~\bfnm{Emad}\binits{E.}} \AND
  \bauthor{\bsnm{Theil},~\bfnm{Florian}\binits{F.}}
(\byear{2009}).
\btitle{A gentle stochastic thermostat for molecular dynamics}.
\bjournal{Journal of Statistical Physics}
\bvolume{135}
\bpages{261--277}.
\end{barticle}
\endbibitem

\bibitem[\protect\citeauthoryear{Li}{2018}]{LiG}
\begin{barticle}[author]
\bauthor{\bsnm{Li},~\bfnm{Wuchen}\binits{W.}}
(\byear{2018}).
\btitle{Geometry of Probability Simplex via Optimal Transport}.
\bjournal{arXiv:1803.06360 [math]}.
\end{barticle}
\endbibitem

\bibitem[\protect\citeauthoryear{Li, Lin and Mont{\'u}far}{2019}]{LiM2}
\begin{barticle}[author]
\bauthor{\bsnm{Li},~\bfnm{Wuchen}\binits{W.}},
  \bauthor{\bsnm{Lin},~\bfnm{Alex}\binits{A.}} \AND
  \bauthor{\bsnm{Mont{\'u}far},~\bfnm{Guido}\binits{G.}}
(\byear{2019}).
\btitle{Affine Natural Proximal Learning}.
\end{barticle}
\endbibitem

\bibitem[\protect\citeauthoryear{Li and Montufar}{2018}]{LiM}
\begin{barticle}[author]
\bauthor{\bsnm{Li},~\bfnm{Wuchen}\binits{W.}} \AND
  \bauthor{\bsnm{Montufar},~\bfnm{Guido}\binits{G.}}
(\byear{2018}).
\btitle{Natural Gradient via Optimal Transport}.
\bjournal{arXiv:1803.07033 [cs, math]}.
\end{barticle}
\endbibitem

\bibitem[\protect\citeauthoryear{Liu and Wang}{2016}]{Liu2016}
\begin{binproceedings}[author]
\bauthor{\bsnm{Liu},~\bfnm{Qiang}\binits{Q.}} \AND
  \bauthor{\bsnm{Wang},~\bfnm{Dilin}\binits{D.}}
(\byear{2016}).
\btitle{{Stein} variational gradient descent: A general purpose bayesian
  inference algorithm}.
In \bbooktitle{Advances In Neural Information Processing Systems}
\bpages{2378--2386}.
\end{binproceedings}
\endbibitem

\bibitem[\protect\citeauthoryear{Lu, Lu and Nolen}{2018}]{lu2018scaling}
\begin{barticle}[author]
\bauthor{\bsnm{Lu},~\bfnm{Jianfeng}\binits{J.}},
  \bauthor{\bsnm{Lu},~\bfnm{Yulong}\binits{Y.}} \AND
  \bauthor{\bsnm{Nolen},~\bfnm{James}\binits{J.}}
(\byear{2018}).
\btitle{Scaling limit of the {Stein} variational gradient descent part I: the
  mean field regime}.
\bjournal{arXiv preprint arXiv:1805.04035}.
\end{barticle}
\endbibitem

\bibitem[\protect\citeauthoryear{Machlup and Onsager}{1953}]{OnM53}
\begin{barticle}[author]
\bauthor{\bsnm{Machlup},~\bfnm{S.}\binits{S.}} \AND
  \bauthor{\bsnm{Onsager},~\bfnm{L.}\binits{L.}}
(\byear{1953}).
\btitle{Fluctuations and irreversible process. {II}. {S}ystems with kinetic
  energy}.
\bjournal{Physical Rev. (2)}
\bvolume{91}
\bpages{1512--1515}.
\bmrnumber{0057766}
\end{barticle}
\endbibitem

\bibitem[\protect\citeauthoryear{Majda and Harlim}{2012}]{majda2012filtering}
\begin{bbook}[author]
\bauthor{\bsnm{Majda},~\bfnm{Andrew~J}\binits{A.~J.}} \AND
  \bauthor{\bsnm{Harlim},~\bfnm{John}\binits{J.}}
(\byear{2012}).
\btitle{Filtering complex turbulent systems}.
\bpublisher{Cambridge University Press}.
\end{bbook}
\endbibitem

\bibitem[\protect\citeauthoryear{Markowich and Villani}{2000}]{MarkoVillani}
\begin{barticle}[author]
\bauthor{\bsnm{Markowich},~\bfnm{P.~A.}\binits{P.~A.}} \AND
  \bauthor{\bsnm{Villani},~\bfnm{C.}\binits{C.}}
(\byear{2000}).
\btitle{On the trend to equilibrium for the {F}okker-{P}lanck equation: an
  interplay between physics and functional analysis}.
\bjournal{Mat. Contemp.}
\bvolume{19}
\bpages{1--29}.
\bnote{VI Workshop on Partial Differential Equations, Part II (Rio de Janeiro,
  1999)}.
\bmrnumber{1812873}
\end{barticle}
\endbibitem

\bibitem[\protect\citeauthoryear{Marzouk et~al.}{2016}]{Marzouk2016}
\begin{binbook}[author]
\bauthor{\bsnm{Marzouk},~\bfnm{Youssef}\binits{Y.}},
  \bauthor{\bsnm{Moselhy},~\bfnm{Tarek}\binits{T.}},
  \bauthor{\bsnm{Parno},~\bfnm{Matthew}\binits{M.}} \AND
  \bauthor{\bsnm{Spantini},~\bfnm{Alessio}\binits{A.}}
(\byear{2016}).
\btitle{Sampling via Measure Transport: An Introduction}
In \bbooktitle{Handbook of Uncertainty Quantification}
\bpages{1--41}.
\bpublisher{Springer International Publishing}, \baddress{Cham}.
\bdoi{10.1007/978-3-319-11259-6_23-1}
\end{binbook}
\endbibitem

\bibitem[\protect\citeauthoryear{Mattingly
  et~al.}{2012}]{mattingly2012diffusion}
\begin{barticle}[author]
\bauthor{\bsnm{Mattingly},~\bfnm{Jonathan~C}\binits{J.~C.}},
  \bauthor{\bsnm{Pillai},~\bfnm{Natesh~S}\binits{N.~S.}},
  \bauthor{\bsnm{Stuart},~\bfnm{Andrew~M}\binits{A.~M.}} \betal{et~al.}
(\byear{2012}).
\btitle{Diffusion limits of the random walk {Metropolis} algorithm in high
  dimensions}.
\bjournal{The Annals of Applied Probability}
\bvolume{22}
\bpages{881--930}.
\end{barticle}
\endbibitem

\bibitem[\protect\citeauthoryear{McLachlan, Quispel and
  Robidoux}{1999}]{mclachlan1999geometric}
\begin{barticle}[author]
\bauthor{\bsnm{McLachlan},~\bfnm{Robert~I}\binits{R.~I.}},
  \bauthor{\bsnm{Quispel},~\bfnm{GRW}\binits{G.}} \AND
  \bauthor{\bsnm{Robidoux},~\bfnm{Nicolas}\binits{N.}}
(\byear{1999}).
\btitle{Geometric integration using discrete gradients}.
\bjournal{Philosophical Transactions of the Royal Society of London. Series A:
  Mathematical, Physical and Engineering Sciences}
\bvolume{357}
\bpages{1021--1045}.
\end{barticle}
\endbibitem

\bibitem[\protect\citeauthoryear{Mielke, Peletier and
  Renger}{2016}]{Mielke16Ons}
\begin{barticle}[author]
\bauthor{\bsnm{Mielke},~\bfnm{A.}\binits{A.}},
  \bauthor{\bsnm{Peletier},~\bfnm{M.~A.}\binits{M.~A.}} \AND
  \bauthor{\bsnm{Renger},~\bfnm{M.}\binits{M.}}
(\byear{2016}).
\btitle{A generalization of Onsager's reciprocity relations to gradient flows
  with nonlinear mobility}.
\bjournal{Journal of Non-Equilibrium Thermodynamics}
\bvolume{41}.
\end{barticle}
\endbibitem

\bibitem[\protect\citeauthoryear{N\"usken and Reich}{2019}]{reichnew}
\begin{barticle}[author]
\bauthor{\bsnm{N\"usken},~\bfnm{Nik}\binits{N.}} \AND
  \bauthor{\bsnm{Reich},~\bfnm{Sebastian.}\binits{S.}}
(\byear{2019}).
\btitle{{Note on Interacting Langevin Diffusions: Gradient Structure and
  Ensemble Kalman Smoother by Garbuno-Inigo, Hoffmann, Li and Stuart}}.
\bjournal{arXiv:1908.}
\end{barticle}
\endbibitem

\bibitem[\protect\citeauthoryear{Oliver, Reynolds and
  Liu}{2008}]{oliver2008inverse}
\begin{bbook}[author]
\bauthor{\bsnm{Oliver},~\bfnm{Dean~S}\binits{D.~S.}},
  \bauthor{\bsnm{Reynolds},~\bfnm{Albert~C}\binits{A.~C.}} \AND
  \bauthor{\bsnm{Liu},~\bfnm{Ning}\binits{N.}}
(\byear{2008}).
\btitle{Inverse theory for petroleum reservoir characterization and history
  matching}.
\bpublisher{Cambridge University Press}.
\end{bbook}
\endbibitem

\bibitem[\protect\citeauthoryear{Ollivier}{2017}]{Ollivier2017_online}
\begin{barticle}[author]
\bauthor{\bsnm{Ollivier},~\bfnm{Yann}\binits{Y.}}
(\byear{2017}).
\btitle{Online {{Natural Gradient}} as a {{Kalman Filter}}}.
\bjournal{arXiv:1703.00209 [math, stat]}.
\end{barticle}
\endbibitem

\bibitem[\protect\citeauthoryear{Onsager}{1931a}]{Ons31p1}
\begin{barticle}[author]
\bauthor{\bsnm{Onsager},~\bfnm{Lars}\binits{L.}}
(\byear{1931}a).
\btitle{Reciprocal Relations in Irreversible Processes. I.}
\bjournal{Phys. Rev.}
\bvolume{37}
\bpages{405--426}.
\bdoi{10.1103/PhysRev.37.405}
\end{barticle}
\endbibitem

\bibitem[\protect\citeauthoryear{Onsager}{1931b}]{Ons31p2}
\begin{barticle}[author]
\bauthor{\bsnm{Onsager},~\bfnm{Lars}\binits{L.}}
(\byear{1931}b).
\btitle{Reciprocal Relations in Irreversible Processes. II.}
\bjournal{Phys. Rev.}
\bvolume{38}
\bpages{2265--2279}.
\bdoi{10.1103/PhysRev.38.2265}
\end{barticle}
\endbibitem

\bibitem[\protect\citeauthoryear{{\"O}ttinger}{2005}]{Ott05}
\begin{bbook}[author]
\bauthor{\bsnm{{\"O}ttinger},~\bfnm{H.~C.}\binits{H.~C.}}
(\byear{2005}).
\btitle{Beyond Equilibrium Thermodynamics}.
\bpublisher{Wiley}.
\end{bbook}
\endbibitem

\bibitem[\protect\citeauthoryear{Otto}{2001}]{otto2001}
\begin{barticle}[author]
\bauthor{\bsnm{Otto},~\bfnm{Felix}\binits{F.}}
(\byear{2001}).
\btitle{The Geometry of Dissipative Evolution Equations the Porous Medium
  Equation}.
\bjournal{Communications in Partial Differential Equations}
\bvolume{26}
\bpages{101-174}.
\end{barticle}
\endbibitem

\bibitem[\protect\citeauthoryear{Ottobre and
  Pavliotis}{2011}]{ottobre2011asymptotic}
\begin{barticle}[author]
\bauthor{\bsnm{Ottobre},~\bfnm{M}\binits{M.}} \AND
  \bauthor{\bsnm{Pavliotis},~\bfnm{GA}\binits{G.}}
(\byear{2011}).
\btitle{Asymptotic analysis for the generalized Langevin equation}.
\bjournal{Nonlinearity}
\bvolume{24}
\bpages{1629}.
\end{barticle}
\endbibitem

\bibitem[\protect\citeauthoryear{Pareschi and
  Toscani}{2013}]{PareschiToscani_book}
\begin{bbook}[author]
\bauthor{\bsnm{Pareschi},~\bfnm{Lorenzo}\binits{L.}} \AND
  \bauthor{\bsnm{Toscani},~\bfnm{Giuseppe}\binits{G.}}
(\byear{2013}).
\btitle{{Interacting Multiagent Systems: Kinetic equations and Monte Carlo
  methods}}.
\bseries{OUP Catalogue}
\bvolume{9780199655465}.
\bpublisher{Oxford University Press}.
\end{bbook}
\endbibitem

\bibitem[\protect\citeauthoryear{Pathiraja and
  Reich}{2019}]{pathiraja2019discrete}
\begin{barticle}[author]
\bauthor{\bsnm{Pathiraja},~\bfnm{Sahani}\binits{S.}} \AND
  \bauthor{\bsnm{Reich},~\bfnm{Sebastian}\binits{S.}}
(\byear{2019}).
\btitle{Discrete gradients for computational {Bayesian} inference}.
\bjournal{Computational Dynamics, to appear; arXiv:1903.00186}.
\end{barticle}
\endbibitem

\bibitem[\protect\citeauthoryear{Pavliotis}{2014}]{pavliotis2014stochastic}
\begin{bbook}[author]
\bauthor{\bsnm{Pavliotis},~\bfnm{Grigorios~A}\binits{G.~A.}}
(\byear{2014}).
\btitle{Stochastic processes and applications: diffusion processes, the
  Fokker-Planck and Langevin equations}
\bvolume{60}.
\bpublisher{Springer}.
\end{bbook}
\endbibitem

\bibitem[\protect\citeauthoryear{Pillai, Stuart and
  Thi{\'e}ry}{2014}]{pillai2014noisy}
\begin{barticle}[author]
\bauthor{\bsnm{Pillai},~\bfnm{Natesh~S}\binits{N.~S.}},
  \bauthor{\bsnm{Stuart},~\bfnm{Andrew~M}\binits{A.~M.}} \AND
  \bauthor{\bsnm{Thi{\'e}ry},~\bfnm{Alexandre~H}\binits{A.~H.}}
(\byear{2014}).
\btitle{Noisy gradient flow from a random walk in Hilbert space}.
\bjournal{Stochastic Partial Differential Equations: Analysis and Computations}
\bvolume{2}
\bpages{196--232}.
\end{barticle}
\endbibitem

\bibitem[\protect\citeauthoryear{Reich}{2011}]{reich2011dynamical}
\begin{barticle}[author]
\bauthor{\bsnm{Reich},~\bfnm{Sebastian}\binits{S.}}
(\byear{2011}).
\btitle{A dynamical systems framework for intermittent data assimilation}.
\bjournal{BIT Numerical Mathematics}
\bvolume{51}
\bpages{235--249}.
\end{barticle}
\endbibitem

\bibitem[\protect\citeauthoryear{Reich}{2013}]{reich2013}
\begin{barticle}[author]
\bauthor{\bsnm{Reich},~\bfnm{Sebastian}\binits{S.}}
(\byear{2013}).
\btitle{A nonparametric ensemble transform method for {Bayesian} inference}.
\bjournal{SIAM Journal on Scientific Computing}
\bvolume{35}
\bpages{A2013--A2024}.
\end{barticle}
\endbibitem

\bibitem[\protect\citeauthoryear{Reich}{2018}]{reich2018data}
\begin{barticle}[author]
\bauthor{\bsnm{Reich},~\bfnm{Sebastian}\binits{S.}}
(\byear{2018}).
\btitle{Data Assimilation-The {S}chr\"odinger Perspective}.
\bjournal{arXiv preprint arXiv:1807.08351}.
\end{barticle}
\endbibitem

\bibitem[\protect\citeauthoryear{Reich and
  Cotter}{2015}]{reich2015probabilistic}
\begin{bbook}[author]
\bauthor{\bsnm{Reich},~\bfnm{Sebastian}\binits{S.}} \AND
  \bauthor{\bsnm{Cotter},~\bfnm{Colin}\binits{C.}}
(\byear{2015}).
\btitle{Probabilistic forecasting and {Bayesian} data assimilation}.
\bpublisher{Cambridge University Press}.
\end{bbook}
\endbibitem

\bibitem[\protect\citeauthoryear{Roberts and
  Rosenthal}{1998}]{roberts1998optimal}
\begin{barticle}[author]
\bauthor{\bsnm{Roberts},~\bfnm{Gareth~O}\binits{G.~O.}} \AND
  \bauthor{\bsnm{Rosenthal},~\bfnm{Jeffrey~S}\binits{J.~S.}}
(\byear{1998}).
\btitle{Optimal scaling of discrete approximations to Langevin diffusions}.
\bjournal{Journal of the Royal Statistical Society: Series B (Statistical
  Methodology)}
\bvolume{60}
\bpages{255--268}.
\end{barticle}
\endbibitem

\bibitem[\protect\citeauthoryear{Roberts et~al.}{2001}]{roberts2001optimal}
\begin{barticle}[author]
\bauthor{\bsnm{Roberts},~\bfnm{Gareth~O}\binits{G.~O.}},
  \bauthor{\bsnm{Rosenthal},~\bfnm{Jeffrey~S}\binits{J.~S.}} \betal{et~al.}
(\byear{2001}).
\btitle{Optimal scaling for various {Metropolis-Hastings} algorithms}.
\bjournal{Statistical science}
\bvolume{16}
\bpages{351--367}.
\end{barticle}
\endbibitem

\bibitem[\protect\citeauthoryear{Roberts et~al.}{1997}]{roberts1997}
\begin{barticle}[author]
\bauthor{\bsnm{Roberts},~\bfnm{Gareth~O}\binits{G.~O.}},
  \bauthor{\bsnm{Gelman},~\bfnm{Andrew}\binits{A.}},
  \bauthor{\bsnm{Gilks},~\bfnm{Walter~R}\binits{W.~R.}} \betal{et~al.}
(\byear{1997}).
\btitle{Weak convergence and optimal scaling of random walk {Metropolis}
  algorithms}.
\bjournal{The Annals of Applied Probability}
\bvolume{7}
\bpages{110--120}.
\end{barticle}
\endbibitem

\bibitem[\protect\citeauthoryear{Schillings and
  Stuart}{}]{schillings2018convergence}
\begin{barticle}[author]
\bauthor{\bsnm{Schillings},~\bfnm{Claudia}\binits{C.}} \AND
  \bauthor{\bsnm{Stuart},~\bfnm{Andrew~M}\binits{A.~M.}}
\btitle{Convergence analysis of ensemble {Kalman} inversion: the linear, noisy
  case}.
\bjournal{Applicable Analysis}
\bvolume{97}
\bpages{107--123}.
\end{barticle}
\endbibitem

\bibitem[\protect\citeauthoryear{Schillings and
  Stuart}{2017}]{schillings2017analysis}
\begin{barticle}[author]
\bauthor{\bsnm{Schillings},~\bfnm{Claudia}\binits{C.}} \AND
  \bauthor{\bsnm{Stuart},~\bfnm{Andrew~M}\binits{A.~M.}}
(\byear{2017}).
\btitle{Analysis of the ensemble {Kalman} filter for inverse problems}.
\bjournal{SIAM Journal on Numerical Analysis}
\bvolume{55}
\bpages{1264--1290}.
\end{barticle}
\endbibitem

\bibitem[\protect\citeauthoryear{Schneider et~al.}{2017}]{schneider2017earth}
\begin{barticle}[author]
\bauthor{\bsnm{Schneider},~\bfnm{Tapio}\binits{T.}},
  \bauthor{\bsnm{Lan},~\bfnm{Shiwei}\binits{S.}},
  \bauthor{\bsnm{Stuart},~\bfnm{Andrew}\binits{A.}} \AND
  \bauthor{\bsnm{Teixeira},~\bfnm{Jo{\~a}o}\binits{J.}}
(\byear{2017}).
\btitle{Earth system modeling 2.0: A blueprint for models that learn from
  observations and targeted high-resolution simulations}.
\bjournal{Geophysical Research Letters}
\bvolume{44}.
\end{barticle}
\endbibitem

\bibitem[\protect\citeauthoryear{Spiegelhalter
  et~al.}{2002}]{spiegelhalter2002bayesian}
\begin{barticle}[author]
\bauthor{\bsnm{Spiegelhalter},~\bfnm{David~J}\binits{D.~J.}},
  \bauthor{\bsnm{Best},~\bfnm{Nicola~G}\binits{N.~G.}},
  \bauthor{\bsnm{Carlin},~\bfnm{Bradley~P}\binits{B.~P.}} \AND
  \bauthor{\bsnm{Van Der~Linde},~\bfnm{Angelika}\binits{A.}}
(\byear{2002}).
\btitle{{Bayesian} measures of model complexity and fit}.
\bjournal{Journal of the Royal Statistical Society: Series B (Statistical
  Methodology)}
\bvolume{64}
\bpages{583--639}.
\end{barticle}
\endbibitem

\bibitem[\protect\citeauthoryear{Sznitman}{1991}]{sznitman}
\begin{binproceedings}[author]
\bauthor{\bsnm{Sznitman},~\bfnm{Alain-Sol}\binits{A.-S.}}
(\byear{1991}).
\btitle{Topics in propagation of chaos}.
In \bbooktitle{Ecole d'Et{\'e} de Probabilit{\'e}s de Saint-Flour XIX --- 1989}
(\beditor{\bfnm{Paul-Louis}\binits{P.-L.}~\bsnm{Hennequin}}, ed.)
\bpages{165--251}.
\bpublisher{Springer Berlin Heidelberg}, \baddress{Berlin, Heidelberg}.
\end{binproceedings}
\endbibitem

\bibitem[\protect\citeauthoryear{Taghvaei et~al.}{2018}]{taghvaei2018kalman}
\begin{barticle}[author]
\bauthor{\bsnm{Taghvaei},~\bfnm{Amirhossein}\binits{A.}},
  \bauthor{\bparticle{de} \bsnm{Wiljes},~\bfnm{Jana}\binits{J.}},
  \bauthor{\bsnm{Mehta},~\bfnm{Prashant~G}\binits{P.~G.}} \AND
  \bauthor{\bsnm{Reich},~\bfnm{Sebastian}\binits{S.}}
(\byear{2018}).
\btitle{{Kalman} filter and its modern extensions for the continuous-time
  nonlinear filtering problem}.
\bjournal{Journal of Dynamic Systems, Measurement, and Control}
\bvolume{140}
\bpages{030904}.
\end{barticle}
\endbibitem

\bibitem[\protect\citeauthoryear{Tong~Lin et~al.}{2018}]{tong2018wasserstein}
\begin{barticle}[author]
\bauthor{\bsnm{Tong~Lin},~\bfnm{Alex}\binits{A.}},
  \bauthor{\bsnm{Li},~\bfnm{Wuchen}\binits{W.}},
  \bauthor{\bsnm{Osher},~\bfnm{Stanley}\binits{S.}} \AND
  \bauthor{\bsnm{Mont{\'u}far},~\bfnm{Guido}\binits{G.}}
(\byear{2018}).
\btitle{Wasserstein Proximal of GANs}.
\end{barticle}
\endbibitem

\bibitem[\protect\citeauthoryear{Toscani}{2006}]{Toscani06Opinion}
\begin{barticle}[author]
\bauthor{\bsnm{Toscani},~\bfnm{Giuseppe}\binits{G.}}
(\byear{2006}).
\btitle{Kinetic models of opinion formation}.
\bjournal{Commun. Math. Sci.}
\bvolume{4}
\bpages{481--496}.
\bmrnumber{2247927}
\end{barticle}
\endbibitem

\bibitem[\protect\citeauthoryear{Villani}{2009}]{vil2008}
\begin{bbook}[author]
\bauthor{\bsnm{Villani},~\bfnm{C\'edric}\binits{C.}}
(\byear{2009}).
\btitle{Optimal Transport: Old and New}.
\bseries{Grundlehren Der Mathematischen {{Wissenschaften}}}
\bvolume{338}.
\bpublisher{{Springer}}, \baddress{Berlin}.
\end{bbook}
\endbibitem

\bibitem[\protect\citeauthoryear{Yang, Mehta and Meyn}{2013}]{yang2013feedback}
\begin{barticle}[author]
\bauthor{\bsnm{Yang},~\bfnm{Tao}\binits{T.}},
  \bauthor{\bsnm{Mehta},~\bfnm{Prashant~G}\binits{P.~G.}} \AND
  \bauthor{\bsnm{Meyn},~\bfnm{Sean~P}\binits{S.~P.}}
(\byear{2013}).
\btitle{Feedback particle filter}.
\bjournal{IEEE transactions on Automatic control}
\bvolume{58}
\bpages{2465--2480}.
\end{barticle}
\endbibitem

\bibitem[\protect\citeauthoryear{Yang, Roberts and
  Rosenthal}{2019}]{yang2019optimal}
\begin{barticle}[author]
\bauthor{\bsnm{Yang},~\bfnm{Jun}\binits{J.}},
  \bauthor{\bsnm{Roberts},~\bfnm{Gareth~O}\binits{G.~O.}} \AND
  \bauthor{\bsnm{Rosenthal},~\bfnm{Jeffrey~S}\binits{J.~S.}}
(\byear{2019}).
\btitle{Optimal Scaling of {Metropolis} Algorithms on General Target
  Distributions}.
\bjournal{arXiv preprint arXiv:1904.12157}.
\end{barticle}
\endbibitem

\end{thebibliography}
\bibliographystyle{imsart-nameyear}

\end{document}